\documentclass{article}
\usepackage[french,english]{babel}
\usepackage[utf8]{inputenc}
\usepackage[T1]{fontenc}
\usepackage{lmodern}
\usepackage{amsmath}
\usepackage{amssymb}
\usepackage{mathrsfs}
\usepackage{amsthm}
\usepackage[shortlabels]{enumitem}
\usepackage{dsfont} 
\usepackage{eurosym} 
\usepackage{stmaryrd} 
\usepackage{eurosym}

\usepackage{titlesec}
\titleclass{\part}{top}
\titleformat{\part}[display]
{\huge\bfseries\centering}{\partname~\thepart}{0pt}{}
\titlespacing*{\part}{0pt}{160pt}{40pt}

\usepackage{tikz}
\usetikzlibrary{arrows, decorations.pathreplacing}

\usepackage{graphicx}
\usepackage{float}

\usepackage{hyperref}  
\hypersetup{                    
	colorlinks=true,                
	breaklinks=true,                
	urlcolor= blue,                 
	linkcolor= black,                
	citecolor= magenta,                
	pdfstartview = FitH,
	bookmarksopen = true               
}

\usepackage{listings}
\usepackage{xcolor}
\usepackage{framed}

\usepackage{fullpage} 

\newcommand{\R}{\mathbb{R}}
\newcommand{\N}{\mathbb{N}}
\newcommand{\E}{\mathbb{E}}
\newcommand{\PP}{\mathbb{P}}

\DeclareMathOperator{\proj}{proj}

\DeclareMathOperator{\TV}{\operatorname{TV}}

\renewcommand{\epsilon}{\varepsilon}

\newcommand{\1}{\mathds{1}}

\newcommand{\mathbbm}[1]{\1}

\title{Stability of the Weak Martingale Optimal Transport Problem}
\makeatletter
\newcommand{\specificthanks}[1]{\@fnsymbol{#1}}
\makeatother
\author{
	M. Beiglböck\footnote{University of Vienna, Austria. Email: \href{mailto:mathias.beiglboeck@univie.ac.at}{\texttt{mathias.beiglboeck@univie.ac.at}},}\textsuperscript{\specificthanks{3},~}\thanks{acknowledges support from FWF through grant no.\ Y00782.}
	\and 
	B. Jourdain\footnote{CERMICS, Ecole des Ponts, INRIA, Marne-la-Vallée, France. E-mails: \href{mailto:benjamin.jourdain@enpc.fr}{\texttt{benjamin.jourdain@enpc.fr}}, \href{mailto:william.margheriti@enpc.fr}{\texttt{william.margheriti@enpc.fr}}} 
	\and
	W. Margheriti\textsuperscript{\specificthanks{2},~}\thanks{acknowledges support from the \textquotedblleft Chaire Risques Financiers\textquotedblright , Fondation du Risque.}
	\and
	G. Pammer\footnote{ETH Zurich, Switzerland. Email: \href{mailto:gudmund.pammer@math.ethz.ch}{\texttt{gudmund.pammer@math.ethz.ch}}}\textsuperscript{\specificthanks{1},~}\thanks{acknowledges support from the Austrian Science Fund (FWF) through grant number W1245.}
}
\date{}

\numberwithin{equation}{section}
\theoremstyle{plain}
\newtheorem{prooff}{Proof}[section]

\addto\captionsfrench{}

\newtheorem{lemma}[prooff]{Lemma}

\newtheorem{theorem}[prooff]{Theorem}
\newtheorem{proposition}[prooff]{Proposition}
\newtheorem{corollary}[prooff]{Corollary}

\theoremstyle{definition}

\newtheorem{rk}[prooff]{Remark}
\newtheorem{remark}[prooff]{Remark}

\newtheorem{definition}[prooff]{Definition}

\theoremstyle{definition}

\theoremstyle{plain}
\newtheorem{demo}[prooff]{Démonstration}

\allowdisplaybreaks

\begin{document}
\maketitle
\begin{abstract}
While many questions in (robust) finance can be posed in the martingale  optimal transport (MOT) framework, others require to consider also non-linear cost functionals. 
Following the terminology of Gozlan, Roberto, Samson and Tetali \cite{GoRoSaTe17} for classical optimal transport, this corresponds to \emph{weak} martingale optimal transport (WMOT).


In this article we establish stability of WMOT 
 which is important since financial data can give only imprecise information on the underlying marginals. As application, we deduce the stability of the superreplication bound for VIX futures as well as  the stability of the stretched Brownian motion and we derive a monotonicity principle for WMOT.

%
%
\end{abstract}


\section{Introduction}

\subsection{Overview}


Gozlan, Roberto, Samson and Tetali \cite{GoRoSaTe17} recently introduced a non-linear relaxation of classical optimal transport. On the one hand, this framework of  \emph{weak optimal  transport} (WOT) still retains many characteristics of usual optimal transport,  allowing for a compelling theory. On the other hand, this type of relaxation is suitable to cover a number of problems from geometric inequalities over  optimal mechanism design to the Schrödinger problem that lie outside the scope of the classical theory, see \cite{GoRoSaSh18, GoJu18, AlBoCh18, BaPa20}.
The same non-linear relaxation is also required in a number of challenges appearing in mathematical finance: this concerns the pricing of VIX futures, stretched Brownian motion, robust pricing for fixed income markets and the optimal Skorokhod embedding problem. Compared to other applications of WOT, the financial applications stick out through the additional martingale condition that results from the no arbitrage assumption prevalent in mathematical finance. This  leads to a  weak martingale optimal transport problem (WMOT) which is the main focus of this article. 

The reason why a transport type problem pops up in finance is that the marginals are fixed based on the celebrated observation of Breeden--Litzenberger \cite{BrLi78} that prices of traded vanilla options fix the marginals of the asset price process $(S_t)$ at the respective maturity times. 

This is only an approximation to reality as we will know only the prices of finitely many derivatives (up to a bid ask spread). To fit the challenges from finance into the WMOT framework, it is thus crucial to establish stability of WMOT w.r.t.\ to the initial data, which is the main aim of this paper.
As particular applications we consider the robust pricing problem for VIX options and the stretched Brownian motion and we will derive a monotonicity principle for the WMOT problem.
These results will be presented in some detail in the remainder of the introduction, while the corresponding proofs are postponed to later sections. 

\subsection{WMOT-framework and main result}

The main contribution of this article is to establish stability for the weak and weak martingale optimal transport, extending and unifying previously known stability results in MOT and WOT. To state our results rigorously, we need to introduce some concepts. 

Let $X$ and $Y$ be  Polish spaces  endowed with the compatible and complete metrics $d_X$ and $d_Y$. Let $\mu$ be in the set $\mathcal P(X)$ of probability measures on $X$, $\nu\in\mathcal P(Y)$ and $C:X\times\mathcal P(Y)\to\R_+$ be a nonnegative measurable function which is convex in the second argument.\footnote{Convexity of $C$ in the second argument is a standing assumption in weak transport theory, in fact the theory for the general  case can be reduced  to the convex instance, see \cite[Theorem 3.7]{AcBePa20}. } We denote by $\Pi(\mu,\nu)$ the set of couplings between $\mu$ and $\nu$, that is $\pi\in\Pi(\mu,\nu)$ if and only if $\pi\in\mathcal P(X\times Y)$ is such that for any measurable subsets $A\subset X$ and $B\subset Y$, $\pi(A\times Y)=\mu(A)$ and $\pi(X\times B)=\nu(B)$. Then the WOT problem consists in the minimisation of
\begin{equation}\label{WOT2}
\tag{WOT}
V_C(\mu,\nu):=\inf_{\pi\in\Pi(\mu,\nu)}\int_XC(x,\pi_x)\,\mu(dx),
\end{equation}
where for all $\pi\in\Pi(\mu,\nu)$, $(\pi_x)_{x\in\R}$ denotes a disintegration of $\pi$ with respect to its first marginal, which we write $\pi(dx,dy)=\mu(dx)\,\pi_x(dy)$, or with a slight abuse of notation, $\pi=\mu\otimes\pi_x$ if the context is not ambiguous. Note that for a measurable map $c:X\times Y\to\R_+$, the WOT problem with the cost function $C:(x,p)\mapsto\int_Yc(x,y)\,p(dy)$ linear in the measure argument amounts to the classical Optimal Transport (OT) problem.

To recall some basic results on existence, duality and stability of the WOT problem, we introduce the appropriate topologies of the underlying spaces. We fix $r\ge1$ and $x_0$, $y_0$ two arbitrary elements of $X$ and $Y$ respectively, their specific value having no impact on our study. Let $\mathcal P^r(X)$ denote the set of all probability measures on $X$ with finite $r$-th moment, i.e.\ $
\mathcal P^r(X)=\left\{p\in\mathcal P(X)\colon\int_Xd_X^r(x,x_0)\,p(dx)<+\infty\right\}.$
Let $\mathcal C(X)$ denote the set of all real-valued continuous functions on $X$. The set $\mathcal P^r(X)$ is equipped with the weak topology induced by
$
\Phi^r(X)=\left\{f\in {\mathcal C}(X)\colon \exists\alpha>0,\ \forall x\in X,\ \vert f(x)\vert\le\alpha(1+d_X^r(x,x_0))\right\}.
$
A compatible metric 
is given by the $r$-Wasserstein distance
$\mathcal W_r$, see \cite{AmGi13,Sa15,Vi1,Vi09} for more details.

In analogy to the classical transport problem it is natural to assume that the cost function $C:X\times\mathcal P^r(Y)\to\R_+$ is \emph{lower semicontinuous}. As in the classical case, the WOT problem then admits a minimiser and a natural  duality relation holds, see \cite{BaBePa18}. 
As in the case of classical optimal transport, stability is a more delicate question; in \cite[Theorem 1.3]{BaPa19}, it is established for continuous cost functions satisfying an appropriate growth condition. 


Motivated by the  model-independent pricing problem, the martingale optimal transport  (MOT) has received considerable attention, see \cite{HoNe12, GaHeTo13, BeHePe12, DoSo12,  ChKiPrSo20}. We also refer to \cite{ObSi17, DeTo17, GhKiLi19} for the multi-dimensional case and to \cite{BeCoHu14, BeNuSt19} for connections to Skorokhod problem. Stability of MOT, that is the continuous dependence of value and optimiser  on the marginals, is a crucial property in light of its numerical resolution and the imperfect data coming available in financial applications and was recently established in \cite{GuOb19,BaPa19,Wi20}.

In one time step, the MOT problem reads as
\begin{equation}
	\label{MOT} \tag{MOT}
	\inf_{\pi \in \Pi_M(\mu,\nu)} \int_{\R \times \R} c(x,y) \, \pi(dx,dy),
\end{equation}
where $\mu,\nu \in \mathcal P^1(\R)$, $c\colon \R \times \R \to \R_+$ is a measurable cost function, and
\[ \Pi_M(\mu,\nu) := \left\{ \pi = \mu \times \pi_x \in \Pi(\mu,\nu) \colon \int_{\R} y \, \pi_x(dy) = x \, \mu\text{-a.s.} \right\} \]
The martingale constraint reflects the condition for a financial market with vanishing interest rates to be  free of arbitrage. According to Strassen's theorem, $\Pi_M(\mu,\nu)$ is not empty if and only if $\mu$ is smaller than $\nu$ in the convex order in the following sense :
$$\forall \varphi:\R\to\R\mbox{ convex },\;\int_\R\varphi(x)\mu(dx)\le\int_\R\varphi(y)\nu(dy),$$
which we also write $\mu\le_c\nu$.

Different problems in finance require to consider also the `weak' version of this problem (WMOT) which permits costs $C$ of the form $C\colon \R \times \mathcal P(\R) \to \R_+$.
The WMOT problem consists in the minimisation of
\begin{equation}\label{WMOT2}
\tag{WMOT}
V^M_C(\mu,\nu):=\inf_{\pi\in\Pi_M(\mu,\nu)}\int_{\R}C(x,\pi_x)\,\mu(dx).
\end{equation}

For a measurable map $c:\R \times \R \to \R_+$, the WMOT problem associated to the cost function $C:(x,p)\mapsto\int_\R c(x,y)\,p(dy)$ linear in the measure argument amounts to an ordinary MOT problem.

 Formally, WMOT is of course captured by  WOT if one sets $C(x,p)$ to be  $\infty$ if the barycenter of $p$ does not equal $x$.
 However this has previously not found much use as this cost functions does not fall in the realm of the typical regularity assumptions.
 Specifically, stability of WOT and MOT is established in two separate results in \cite{BaPa19}.
 In the present article, we unify and extend these results, in particular we will prove the following:

\begin{theorem}[Stability]\label{thm:2Polish2easy} Assume that $C:\R\times\mathcal P^1(\R)\to\R$ is continuous, convex in the second argument and that there exists a  constant $K>0$ such that for $(x,p)\in \R\times\mathcal P^1(\R)$
	\begin{equation}\label{CgrowthexponentrWOT2easy}
	\vert C(x,p)\vert\le K\left(|x|+\int_Y |y|\,p(dy)\right).
	\end{equation}
	For $k\in\N$, let $\mu^k, \nu^k\in\mathcal P^1(\R)$,  (resp. satisfying $\mu^k\le_c\nu^k$) converge in $\mathcal P^1(\R)$ to $\mu$ and $\nu$, respectively. 	Then
        $$\lim_{k\to\infty}V_C(\mu^k,\nu^k)=V_C(\mu,\nu)\;\;\left(\mbox{resp. }\lim_{k\to\infty}V^M_C(\mu^k,\nu^k)=V^M_C(\mu,\nu)\right).$$
      There exist minimisers $\pi^{k,*}\in\Pi(\mu^k,\nu^k)$ (resp. $\pi^{k,*}\in\Pi_M(\mu^k,\nu^k)$) of $V_C(\mu^k,\nu^k)$ (resp. $V^M_C(\mu^k,\nu^k)$) and any accumulation point of $(\pi^{k,*})_{k\in\N}$ for the weak convergence topology is a minimiser of $V_C(\mu,\nu)$ (resp. $V^M_C(\mu,\nu)$).
	\end{theorem}

In Section \ref{subsec:stability} we will present several stronger and more general versions of Theorem \ref{thm:2Polish2easy}. Growth- and continuity assumptions may be weakened and in the case of WOT, the results apply to abstract Polish spaces. \textcolor{black}{In contrast to Theorem \ref{thm:2Polish2easy}, according to an intriguing counter-example given by Br\"uckerhoff and Juillet \cite{BrJu21}, stability fails for MOT and, as a consequence, for the more general WMOT problem in $\R^d$ with $d\ge 2$. That is why all the statements concerning those problems only address the one-dimensional case. We note that in view of financial applications, this represents no restriction.}

\subsection{Robust pricing of VIX-futures}

The domain of robust mathematical finance distinctly surpasses the one of MOT. As example consider the Volatility Index (VIX), often referred to as the Fear Index, which is a popular measure to determine market sentiment.
When investors expect the market to move vigorously, they typically tend to purchase more options, which has an impact on implied volatility levels.
The VIX is by definition the implied volatility calculated on a 30 days horizon on the S\&P 500.
The more the VIX increases, the more demand is expressed for options, which become more expensive.
In that case the market is described as volatile.
Conversely, a decreasing VIX often means less demand and therefore decreasing option prices, hence the market is perceived as calm.

We consider 
the S\&P 500 $(S_t)_{t\in\{T_1,T_2\}}$, tradable at dates $T_1$ and $T_2=T_1+30\text{ days}$.
We suppose known the market price of call options for any strike $K\ge0$, so that by the Breeden-Litzenberger formula \cite{BrLi78} we get the respective probability distributions $\mu$ and $\nu$ of $S_{T_1}$ and $S_{T_2}$.
We allow trading at time $0$ in vanilla options with maturities $T_1$ and $T_2$, and trading at time $T_1$ in the S\&P 500 and the forward-starting log-contract, that is the option with payoff $\frac{-2}{T_2-T_1}\ln\frac{S_{T_2}}{S_{T_1}}$ at $T_2$.
In this setting, Guyon, Menegaux and Nutz \cite{GuMeNu17} derive the model-independent arbitrage-free upper bound for the VIX future expiring at $T_1$, given by the smallest superreplication price at time $0$
\begin{equation}
\label{eq:primal VIX}
P_{\textrm{super}}(\mu,\nu) = \inf \{\mu(u_1) + \nu(u_2)\},
\end{equation}
where the infimum is taken over all $(u_1,u_2)\in L^1(\mu)\times L^1(\nu)$ and measurable maps $\Delta^S,\Delta^L$ such that for all $(x,y,v)\in(0,+\infty)^2\times[0,+\infty)$,
\begin{equation}\label{replicationCondition}
u_1(x) + u_2(y) +\Delta^S(x,v)(y- x) + \Delta^L(x,v)\left(-\frac{2}{T_2-T_1} \ln \frac{y}{x} -v^2 \right)-v \ge 0.
\end{equation}

Note that the primal problem $P_{\textrm{super}}(\mu,\nu)$ given by \eqref{replicationCondition} involves three variables $x,y,v$, which stand respectively for the S\&P 500 at time $T_1$, the S\&P 500 at time $T_2$, and the VIX at time $T_1$.
Naturally we would then expect to find three marginals in the dual formulation.
Strikingly, in \cite[Proposition 4.10]{GuMeNu17} the dual side of the superreplication of the VIX is reformulated as a WMOT problem
.
\begin{proposition}[Guyon, Menegaux, Nutz, 2017]
	\label{prop:VIX}
	Let $0<T_1<T_2$.
	Let $\mu,\nu$ be probability measures on $(0,+\infty)$, in the convex order, and which finitely integrate $|\ln(x)| + |x|$.
	Then the dual problem $D_{super}$ consists of
	\begin{equation}
	\label{eq:dual VIX}
	D_{\textrm{super}}(\mu,\nu) =
	\sup_{\pi \in \Pi_M(\mu,\nu)} \int_{(0,+\infty)} \sqrt{\frac{2}{T_2-T_1} \int_{(0,+\infty)} \ln \left( \frac{x}{y} \right) \, \pi_x(dy) }\, \mu(dx),
	\end{equation}
	The values of $P_{\textrm{super}}(\mu,\nu)$ and $D_{\textrm{super}}(\mu,\nu)$ coincide.
\end{proposition}

We contribute to the theory by establishing stability of \eqref{eq:dual VIX}:

\begin{theorem}
	\label{thm:VIX}
	In the setting of Proposition \ref{prop:VIX}, for $k\in \N$ let each pair $(\mu^k,\nu^k),(\mu,\nu)$ of probability measures on $(0,+\infty)$ be in the convex order and finitely integrate $f(x) := |\ln(x)| + |x|$.
	For $k\to+\infty$, let $(\mu^k,\nu^k)$ converge weakly to $(\mu,\nu)$ and
	\[
		\mu^k(f) \to \mu(f), \quad \nu^k(f) \to \nu(f).
	\]
	\textcolor{black}{Then 
	\[	
		\lim_{k \to +\infty} D_{\textrm{super}}(\mu^k,\nu^k) = D_{\textrm{super}}(\mu,\nu),
	\]
	for each $k\in\N$, there exist in $\Pi_M(\mu^k,\nu^k)$ maximisers of $D_{\textrm{super}}(\mu^k,\nu^k)$, and any weak accumulation point of a sequence $(\pi^{k,\ast})_{k \in \N}$ of such maximizers maximizes $D_{\textrm{super}}(\mu,\nu)$.}
\end{theorem}

\subsection{Stretched Brownian motion}

In \cite{BaBeHuKa20} a Benamou-Brenier type formulation for MOT is suggested.
In dimension one, this problem consists in maximising for two probabilities $\mu,\nu$ in the convex order
\begin{equation}\label{MBB}
\tag{MBB}
MT(\mu,\nu):=\sup\E\left[\int_0^1\sigma_t\,dt\right]
\end{equation}
over all filtered probability spaces $(\Omega,(\mathcal F_t)_{t\in[0,1]},\PP)$, real-valued $(\mathcal F_t)_{t\in[0,1]}$-progressive process $(\sigma_t)_{t\in[0,1]}$ and real-valued $(\mathcal F_t)_{t\in[0,1]}$-Brownian motions $(B_t)_{t\in[0,1]}$ such that the process
\[
(M_t)_{t\in[0,1]}=\left(M_0+\int_0^t\sigma_s\,dB_s\right)_{t\in[0,1]}
\]
is a continuous martingale which satisfies $M_0\sim\mu$ and $M_1\sim\nu$. When the second moment of $\nu$ is finite, then \eqref{MBB} has a unique maximiser $(M^*_t)_{t\in[0,1]}$ \cite[Theorem 1.5]{BaBeHuKa20} called the stretched Brownian motion from $\mu$ to $\nu$, since it is the martingale subject to the constraints $M^*_0\sim\mu$ and $M^*_1\sim\nu$ which correlates the most with the Brownian motion.

Let $C_2:\R\times\mathcal P_2(\R)\to\R$ be defined for all $(x,p)\in\R\times\mathcal P_2(\R)$ by $C_2(x,p)=\mathcal W_2^2(p,\mathcal N(0,1))$, where $\mathcal N(0,1)$ denotes the unidimensional standard normal distribution. Let $\mu,\nu\in\mathcal P_2(\R)$ be in the convex order and $V^M_{C_2}(\mu,\nu)$ be the value function given by \eqref{WMOT2} for the cost function $C_2$. Let $\pi^*\in\Pi_M(\mu,\nu)$ be optimal for $V^M_{C_2}(\mu,\nu)$ and $M^*$ be the stretched Brownian motion from $\mu$ to $\nu$. Then Remark 2.1, Theorem 2.2 and Remark 2.3 from \cite{BaBeHuKa20} imply that
\begin{enumerate}[label=(\alph*)]
	\item $MT(\mu,\nu)=\frac12\left(1+\int_\R\vert y\vert^2\,\nu(dy)-V^M_{C_2}(\mu,\nu)\right)$;
	\item $\pi^*$ is the joint probability distribution of $(M^*_0,M^*_1)$, and conversely
	\begin{equation}\label{eq:stretched bm solution}
	\forall t\in[0,1],\quad M^*_t=\E\left[F_{\pi^*_X}^{-1}(F_{\mathcal N(0,1)}(B_1))\vert X,(B_s)_{0\le s\le t}\right],
	\end{equation}
	where $X\sim\mu$ is a random variable independent of the Brownian motion $(B_t)_{t\in[0,1]}$, and $F_\eta(x)=\eta((-\infty,x])$, resp. $F_{\eta}^{-1}(u)=\inf\{x\in\R:F_\eta(x)\ge u\}$ denotes the cumulative distribution function, resp.\ the quantile function of a probability distribution $\eta\in\mathcal P(\R)$.
\end{enumerate}


As a consequence of Theorems \ref{thm:2Polish} and \ref{thm:2Polish2}, we obtain the following stability result for the stretched Brownian motion:

\begin{corollary}[Stability of stretched Brownian motion]\label{cor:stretched bm} Let $r\ge2$ and $\mu^k,\nu^k,\mu,\nu\in\mathcal P^r(\R)$, $k\in\N$ be such that for all $k\in\N$, $\mu^k\le_c\nu^k$ and $\mu^k$, resp. $\nu^k$, converges to $\mu$, resp. $\nu$, in $\mathcal W_r$.
	
	For $k\in\N$, let $M^k$ be the stretched Brownian motion from $\mu^k$ to $\nu^k$ and $M^*$ be the stretched Brownian motion from $\mu$ to $\nu$. Equipping ${\mathcal C}([0,1])$ with the supremum distance and denoting by $\mathcal L(Z)$ the law of any random variable $Z$, we have $\lim_{k\to\infty}\mathcal{W}_r\left(\mathcal L((M^k_t)_{t\in[0,1]}),\mathcal L((M^*_t)_{t\in[0,1]})\right)=0$.
\end{corollary}

\subsection{Monotonicity principle}

Recently the notion of martingale $C$-monotonicity in \cite{BaPa19} was introduced for WMOT to show stability of MOT.

\begin{definition}[Martingale $C$-monotonicity]
	We say that a Borel set $\Gamma \subset \R \times \mathcal P^1(\R)
        $ is \emph{martingale $C$-monotone} if for any $N \in \N$, any collection $(x_1,p_1),\ldots,(x_N,p_N) \in \Gamma$, and $q_1,\ldots,q_N \in \mathcal P^1(\R)$ such that $\sum_{i=1}^N p_i = \sum_{i=1}^N q_i$
	and $\int_\R y \, p_i(dy) = \int_\R y \, q_i(dy)$, we have
	\[  \sum_{i=1}^N C(x_i,p_i) \leq \sum_{i=1}^N C(x_i,q_i). \]
\end{definition}

So far, it was known that martingale $C$-monotonicity is a necessary optimality criterion in the following sense, c.f.\ \cite[Theorem 3.4]{BaPa19}:
let $\pi^\ast \in \Pi_M(\mu,\nu)$ be a martingale coupling which minimises \eqref{WMOT2}, then 
there is a martingale $C$-monotone set $\Gamma$ with
\begin{equation} \label{eq:martingale C monotone coupling}
(x,\pi_x) \in \Gamma\quad \text{for $\mu(dx)$-almost every $x$.}
\end{equation}

\begin{remark}\label{rk:finiteSupportOptimal} Conversely, if $\pi\in\Pi_M(\mu,\nu)$ is a finitely supported coupling of the form $\frac1N\sum_{i=1}^N\delta_{x_i}(dx)\,p_i(dy)$ for $x_1<\cdots<x_n\in\R$ and $p_1,\cdots,p_N\in\mathcal P^1(\R)$ and satisfies \eqref{eq:martingale C monotone coupling} for some martingale $C$-monotone set $\Gamma$, then it is optimal. Indeed, in that case $(x_1,p_1),\cdots,(x_N,p_N)\in\Gamma$ and any martingale coupling $\pi'\in\Pi_M(\mu,\nu)$ is of the form $\pi'(dx,dy)=\frac1N\sum_{i=1}^N\delta_{x_i}(dx)\,q_i(dy)$, where $q_1,\cdots,q_N\in\mathcal P^1(\R)$ are such that $\sum_{i=1}^Np_i=\sum_{i=1}^Nq_i$ and for each $i\in\{1,\cdots,N\}$, $\int_\R y\,p_i(dy)=x_i=\int_\R y\,q_i(dy)$. By definition of martingale $C$-monotonicity, we get
	\[
	\int_{\R\times\R}C(x,\pi_x)\,\mu(dx)=\frac1N\sum_{i=1}^NC(x_i,p_i)\le\frac1N\sum_{i=1}^NC(x_i,q_i)=\int_{\R^2}C(x,\pi'_x)\,\mu(dx),
	\]
	hence $\pi$ is optimal.
\end{remark}

However, the question remained open if any martingale coupling satisfying \eqref{eq:martingale C monotone coupling} is optimal.
Our stability results allow us to confirm that this is indeed the case.

\begin{theorem}[Sufficiency] \label{thm:3r}
	Let $r\ge 1$, $\mu,\nu\in\mathcal P^r(\R)$ be in convex order, and 
	$C\colon \R \times \mathcal P^r(\R) \to \R$ be a measurable cost function, 
	continuous in the second argument and such that there exists a finite constant $K$ which satisfies
	\[
	\forall(x,p)\in \R\times\mathcal P^r(\R),\quad C(x,p)\le K\left(1+|x|^r+\int_\R|y|^rp(dy)\right).
	\]
Let $\Gamma$ be martingale $C$-monotone and $\pi \in \Pi_M(\mu,\nu)$ be such that we have \eqref{eq:martingale C monotone coupling}. Then $\pi$ is optimal for \eqref{WMOT2}.
\end{theorem}

In turn Theorem \ref{thm:3} allows us to strengthen \cite[Lemma A.2]{BeJu16} and \cite[Theorem 1.3]{Gr16} by relaxing the required continuity of the cost function to conclude that a martingale coupling is optimal if it is concentrated on a finitely optimal set.
The notion of finite optimality was developed in the spirit of cyclical monotonicity for MOT in \cite{BeJu16} and \cite{Gr16}.
The notion of cyclical monotonicity (see for instance \cite{Vi09}) is a remarkable tool in the theory of OT, which allows to determine optimality of a coupling only by knowing its support.

\begin{definition}[Competitor]
	Let $\alpha=\mu\otimes\alpha_x \in \mathcal P^1(\R \times \R)$.
	We call $\alpha'=\mu'\otimes\alpha'_x \in \mathcal P^1(\R \times \R)$ a competitor of $\alpha$, if
	\[
	\mu=\mu',\quad\textcolor{black}{\int_{x\in\R}\alpha_x(dy)\mu(dx)=\int_{x\in\R}\alpha'_x(dy)\mu'(dx)}\quad\text{and}\quad\int_\R y \, \alpha_x(dy) = \int_\R y \, \alpha'_x(dy), \quad \mu(dx)\text{-a.e.} \]
\end{definition}

\begin{definition}[Finite optimality]
	Let $c \colon \R \times \R\to\R$ be a cost function.
	We say that a Borel set $\tilde \Gamma \subset \R \times \R$ is \emph{finitely optimal} for $c$ if for every probability measure $\alpha \in \mathcal P(\R\times\R)$ finitely supported on $\tilde \Gamma$, we have 
	\[
	\int_{\R\times\R} c(x,y) \, \alpha(dx,dy) \leq \int_{\R\times\R} c(x,y) \,\alpha'(dx,dy),
	\]
	for every competitor $\alpha'$ of $\alpha$.
\end{definition}
\begin{corollary}[Monotonicity principle for MOT] \label{cor:finite optimalityr}
	Let $r\ge 1$, $\mu,\nu\in\mathcal P^r(\R)$ be in convex order,
	$c\colon \R \times \R \to \R$ be measurable and such that $y\mapsto c(x,y)$ is continuous for all $x\in\R$ and $\sup_{(x,y)\in\R^2}\frac{|c(x,y)|}{1+|x|^r+|y|^r}<\infty$.	
	Then $\pi \in \Pi_M(\mu,\nu)$ 
	 is optimal for \eqref{MOT}
	if and only if it
	is 
	concentrated on a finitely optimal set.
      \end{corollary}

\subsection{Organization of the paper}
In Section \ref{sec:main results2} we introduce a number of concepts that are needed later on and state our main results on  the stability of the WOT and the WMOT problems in full strength. 
Section \ref{sec:proofs} is devoted to the proofs. Subsection \ref{sec:stability} consists of the unified proof of the stability of the WOT and the WMOT problems. Subsections \ref{sec:prvix} and \ref{sec:prstr} respectively address the proofs of the stability of the VIX superreplication price and of the stretched Brownian motion. Subsection \ref{sec:MartingaleMonotonicity} consists in showing that martingale $C$-monotonicity is sufficient for optimality for the WMOT problem. 
Finally several auxiliary  lemmas are collected in the Appendix \ref{sec:AppendixApplicationPaper}.

\section{Notation and comprehensive stability result}
\label{sec:main results2}

\subsection{Adapted weak topologies}

The usual weak topology is sometimes not appropriate  to handle settings where the time structure and flow of information play a distinct role.
In particular in the context of mathematical finance, this topology is too weak to provide continuity of sequential decision making problems such as optimal stopping and utility maximization.
Remarkably, in discrete time there exists a canonical extension of the weak topology which adequately captures the temporal structure of stochastic processes, the adapted weak topology \cite{BaBaBeEd19b}.

Let $r \in [1,+\infty)$.
In our setting, the adapted topology (of index $r$) is induced by the adapted Wasserstein distance $\mathcal{AW}_r$ of index $r$ defined for all probabilities $\pi,\pi'\in\mathcal P^r(X\times Y)$ with first marginals $\mu,\mu' \in \mathcal P^r(X)$ by
\[
\mathcal{AW}_r(\pi,\pi'):=\inf_{\chi\in\Pi(\mu,\mu')}\left(\int_{X\times X}\left(d_X^r(x,x')+\mathcal W_r^r(\pi_x,\pi'_{x'})\right)\,\chi(dx,dx')\right)^{\frac1r},
\]
where $(\pi_x)_{x \in X}$ and $(\pi'_x)_{x \in X}$ denote disintegrations of $\pi$ and $\pi'$, respectively, w.r.t.\ the first coordinate.
In the following we may occassionally write $\mu \otimes \pi_x$ in lieu of $\pi$, and use this notation to define probabilities on $X \times Y$.
The adapted Wasserstein distance satisfies
\begin{equation}\label{AWr=Wr}
\mathcal{AW}_r(\pi,\pi')=\mathcal W_r(J(\pi),J(\pi')),
\end{equation}
see for instance \cite{BaBeEdPi17}, where $J$ is the embedding map from $\mathcal P(X\times Y)$ to $\mathcal P(X\times\mathcal P(Y))$, namely
\begin{equation}\label{defJ}
J:\mathcal P(X\times Y)\ni\pi=\mu\otimes\pi_x\mapsto\mu(dx)\,\delta_{\pi_x}(dp)\in\mathcal P(X\times\mathcal P(Y)).
\end{equation}

In the companion paper  \cite{BeJoMaPa21a} we prove that any coupling whose marginals are approximated by probability measures can be approximated by couplings with respect to the adapted Wasserstein distance (see Proposition \ref{prop:1b} below). 

\subsection{An extension of the weak and adapted topologies}

For $r\ge1$, the Wasserstein distance $\mathcal W_r$ is widely used to measure the distance between two probability measures with finite $r$-th moment.
In order to measure the distance between two couplings, one could also use the stronger adapted Wasserstein distance for reasons discussed above. Despite being very handy, those distances sometimes lack topological convenience. For example, the $\mathcal W_r$-balls $\{p\in\mathcal P^r(X)\colon\mathcal W_r(p,\delta_{x_0})\le R\}$, $R>0$, are not compact for the $\mathcal W_r$-distance topology.
Since the proof of sufficiency of martingale $C$-monotonicity (see Section~\ref{sec:MartingaleMonotonicity} below) relies on a compactness argument, we choose in the present paper to work with a finer topology.
We give the definition here as well as some insight to understand its basic properties. All proofs and technical details are deferred to Section~\ref{sec:extensionPrPf} below.


\begin{definition}
	\label{def:Pf}
	Let $f \colon X \to [1,+\infty)$ be continuous. We consider the space 
	\[
	\mathcal P_f(X) = \left\{p \in \mathcal P(X) \colon p(f) < +\infty\right\}.
	\]
	
	We equip $\mathcal P_f(X)$ with the topology induced by the following convergence: a sequence $(p_k)_{k \in \N} \in
	\mathcal P_f(X)^\N$ converges in $\mathcal P_f(X)$ to $p$ if and only if one of the two following equivalent assertions is satisfied:
	\begin{enumerate}[(i)]
		\item $p_k\underset{k\to+\infty}{\longrightarrow}p$ in $\mathcal P(X)$ endowed with the weak convergence topology and $p_k(f)\underset{k\to+\infty}{\longrightarrow}p(f)$.
		\item $p_k(h)\underset{k\to+\infty}{\longrightarrow} p(h)$ for all $h\in \Phi_f(X):= \{h \in \mathcal C(X) \colon \exists\alpha>0,\ \forall x\in X,\ \vert h(x)\vert\le\alpha f(x) \}$ with $\mathcal C(X)$ denoting the space of continuous functions from $X$ to $\R$.
	\end{enumerate}
\end{definition}

Unless explicitly stated otherwise, $\mathcal P(X)$ is endowed with the weak convergence topology; for $r\ge1$, $\mathcal P^r(X)$ is endowed with the $\mathcal W_r$-distance topology; for $f:X\to[1,+\infty)$ continuous, $\mathcal P_f(X)$ is endowed with the topology induced by the convergence introduced in Definition~\ref{def:Pf}. When $f$ is the map $x\mapsto1+d_X^r(x,x_0)$, then $\mathcal P_f(X)=\mathcal P^r(X)$ and the two topologies coincide.
Hence the reader who is not willing to consider this extension may completely disregard it and consistently view $\mathcal P_f(X)$ as the usual Wasserstein space $\mathcal P^r(X)$.

We will mainly address convergences of probability measures from a topological point of view. However it will sometimes prove useful to consider the metric $\overline{\mathcal W}_f$ defined on $\mathcal P_f(X)$ by
\begin{equation}\label{defWbarf}
\forall p,q\in\mathcal P_f(X),\quad\overline{ \mathcal W }_f(p,q) := \sup_{\substack{h \colon X \to [-1,1], \\ h \text{ is
			$1$-Lipschitz}}} (p(fh) - q(fh)),
\end{equation}
which is a complete metric compatible with the topology on $\mathcal P_f(X)$.

A continuous function $g:Y\to[1,+\infty)$ can naturally be lifted to a continuous function $\hat g:\mathcal P_g(Y)\to[1,+\infty)$ by setting
\begin{equation}\label{defhatg}
\forall p\in\mathcal P_g(Y),\quad\hat g(p)=p(g).
\end{equation}

A convenient aspect of this topology is that the spaces $\mathcal P_{\hat g}(\mathcal P(Y))$ and $\mathcal P_{\hat g}(\mathcal P_g(Y))$ and their respective topologies, a priori different, coincide. If moreover $\mathcal P_g(Y)$ is endowed with the metric $\overline{\mathcal W}_g$, then those topological spaces are also equal to $\mathcal P^1(\mathcal P_g(Y))$,  with $(1,\mathcal P_g(Y),\overline{\mathcal W}_g)$ replacing $(r,X,d_X)$. Therefore one can freely switch between the topological spaces $\mathcal P_{\hat g}(\mathcal P(Y))$, $\mathcal P_{\hat g}(\mathcal P_g(Y))$ and $\mathcal P^1(\mathcal P_g(Y))$.

It is also possible to extend the adapted weak topology in the spirit of \eqref{AWr=Wr}. Recall the map $J$ defined by \eqref{defJ} which embeds $\mathcal P(X\times Y)$ into $\mathcal P(X\times\mathcal P(Y))$. For two real-valued functions $f$ and $g$ respectively defined on $X$ and $Y$, we denote by $f\oplus g$ the map $X\times Y\ni(x,y)\mapsto f(x)+g(y)$.

\begin{definition}\label{AWf+gdef} Let $f:X\to[1,+\infty)$ and $g:Y\to[1,+\infty)$ be continuous. For $k\in\N$, let $\mu^k,\mu\in\mathcal P_f(X)$, $\nu^k,\nu\in\mathcal P_g(Y)$, $\pi^k\in\Pi(\mu^k,\nu^k)$ and $\pi\in\Pi(\mu,\nu)$. We say that $(\pi^k)_{k\in\N}$ converges in $\mathcal{AW}_{f\oplus g}$ to $\pi$ if one of the two following equivalent assertions is satisfied:
	\begin{enumerate}[(i)]
		\item\label{it:JpiktoJpif+g} $J(\pi^k)\underset{k\to+\infty}{\longrightarrow}J(\pi)$ in $\mathcal P_{f\oplus\hat g}(X\times\mathcal P(Y))$.
		\item $J(\pi^k)\underset{k\to+\infty}{\longrightarrow}J(\pi)$ in $\mathcal P(X\times\mathcal P(Y))$, $\mu^k(f)\underset{k\to+\infty}{\longrightarrow}\mu(f)$ and $\nu^k(g)\underset{k\to+\infty}{\longrightarrow}\nu(g)$.
	\end{enumerate}
\end{definition}

Again a useful fact is that $\mathcal P_{f\oplus\hat g}(X\times\mathcal P(Y))$ and $\mathcal P_{f\oplus\hat g}(X\times\mathcal P_g(Y))$ and their respective topologies are equal, hence we can rephrase \ref{it:JpiktoJpif+g} as 
\[ J(\pi^k)\underset{k\to+\infty}{\longrightarrow}J(\pi) \text{ in } \mathcal P_{f\oplus\hat g}(X\times\mathcal P_g(Y)). \]
When $f(x)= 1+d_X^r(x,x_0)$ and $g(y)=1+d_Y^r(y,y_0)$, then $(\pi^k)_{k\in\N}$ converges in $\mathcal{AW}_{f\oplus g}$ to $\pi$ if and only if it converges in $\mathcal{AW}_r$. Once again, the reader may skip this extension and consider as she wishes that convergences in $\mathcal{AW}_{f\oplus g}$ mean convergences in $\mathcal{AW}_r$.

\subsection{Stability}\label{subsec:stability}

The continuous dependence of the WOT and WMOT problems, or of optimal transport problems in general, on the marginals is a crucial property:
often, these problems are computationally solvable when the marginals are finitely supported. It is therefore natural to discretise the marginals and solve discretised versions.
This approach works only if we know that the discretised problem sufficiently converges to the original one.
Moreover, in practice marginals are only approximately known and usually derived from noisy data, which again emphasizes the importance of stability.
Therefore we are interested in the continuity of the maps $(\mu,\nu)\mapsto V_C(\mu,\nu)$ and $(\mu,\nu)\mapsto V^M_C(\mu,\nu)$. First we derive the lower semicontinuity of those maps.

Recall that sequence $(\mu^k)_{k\in\N}$ of probability measures on $X$ is said to converge strongly to some $\mu\in\mathcal P(X)$ if and only if for any measurable subset $A\subset X$, $\mu^k(A)$ converges to $\mu(A)$ as $k$ goes to $+\infty$. 
\textcolor{black}{
If $X\subset \R$,  then the Borel $\sigma$-field $\mathcal B(X)$ of $X$ for the induced topology (the coarsest topology which makes the canonical injection $X\ni x\mapsto \iota(x)=x\in\R$ continuous) contains $\{A\cap X: A\in\mathcal B(\R)\}$ by measurability of $\iota$. Since $\{A\cap X: A\in\mathcal B(\R)\}$ is a $\sigma$-field which contains the open subsets of $X$ for the induced topology and therefore $\mathcal B(X)$, one has $\mathcal B(X)=\{A\cap X: A\in\mathcal B(\R)\}$. For $\mu\in{\mathcal P}(X)$, one may define $\tilde \mu\in{\mathcal P}(\R)$ by \begin{equation}\label{defmutildemu}
  \forall A\in\mathcal B(\R),\;\tilde\mu(A)=\mu(A\cap X)
\end{equation}($\tilde \mu$ is the image of $\mu$ by $\iota$). By the characterization of the weak convergence through open sets in the Portmanteau theorem, a sequence $(\mu^k)_{k\in\N}$ converges weakly to $\mu$ in $\mathcal P(X)$ if and only if the sequence $(\tilde \mu^k)_{k\in\N}$ converges weakly to $\tilde \mu$ in ${\mathcal P}(\R)$. When $X\in\mathcal B(\R)$, then $\mathcal B(X)\subset\mathcal B(\R)$ and if $\nu\in\mathcal P(\R)$ is such that $\nu(X)=1$ then the restriction $\nu|_{\mathcal B(X)}$ of $\nu$ to $\mathcal B(X)$ belongs to $\mathcal P(X)$ and is such that $\widetilde{\nu|_{\mathcal B(X)}}=\nu$.   The subset $X\subset \R$ is called a Polish subspace of $\R$ if it Polish for the induced topology. By Alexandrov's theorem, $X$ is a Polish subspace of $\R$ if and only if it is a countable intersection of open subsets of the real line, which ensures that $X\in\mathcal B(\R)$. 
We denote by $\mathcal F^{|\cdot|}(X)$ the set of continuous functions $f \colon X \to [1,+\infty)$ such that $\forall x\in X$, $f(x)\ge|x|$. 
 When $\mu\in\mathcal P_f(X)$ and $\nu\in\mathcal P_g(Y)$ with $X$ and $Y$ Polish subspaces of $\R$ and $f\in\mathcal F^{|\cdot|}(X)$, $g\in\mathcal F^{|\cdot|}(Y)$, then we say that $\mu\le_c\nu$ if the image $\tilde \mu$ of $\nu$ by the injection from $Y$ to $\R$ dominates for the convex order the image $\tilde\nu$ of $\nu$ by the injection from $X$ to $\R$. Then $\Pi_M(\tilde\mu,\tilde\nu)\neq\emptyset$ by Strassen's theorem and for any $\tilde \pi\in \Pi_M(\tilde\mu,\tilde\nu)$ the probability measure defined by $\pi(A)=\tilde\pi(A)$ for all Borel subset of $X\times Y$ belongs to
 $$\Pi_M(\mu,\nu):= \left\{ \pi = \mu \otimes \pi_x \in \Pi(\mu,\nu) \colon \int_{Y} y \, \pi_x(dy) = x \, \mu\text{-a.s.} \right\}.$$
 For a measurable cost function on $X\times\mathcal P_g(Y)$, we define
\begin{equation}\label{WMOT22}
\tag{WMOT}
V^M_C(\mu,\nu):=\inf_{\pi\in\Pi_M(\mu,\nu)}\int_{X}C(x,\pi_x)\,\mu(dx).
\end{equation}
and denote by $\tilde C$ the cost function on $X\times\mathcal P_g(Y)$ defined by \begin{equation}
   \tilde C(x,p)=C(x,p)\1_{\left\{x=\int_Y y \, p(dy)\right\}}+\infty\1_{\left\{x\neq \int_Y y \, p(dy)\right\}}.\label{deftildeC}
 \end{equation}}

In the following statements which address both the WOT and the WMOT problems, namely Theorems \ref{thm:2Polish}, \ref{thm:2upp} and \ref{thm:2Polish2}, the necessary changes in the martingale case are added in parentheses.

\begin{theorem}\label{thm:2Polish} Let $X$ and $Y$ be Polish spaces {\color{black}(resp.\ $X$ and $Y$ Polish subspaces of $\R$)}, $f:X\to[1,+\infty)$ and $g:Y\to[1,+\infty)$ be continuous \textcolor{black}{(resp.\ $f \in \mathcal F^{|\cdot|}(X)$ and $g \in \mathcal F^{|\cdot|}(Y)$)}, $C:X\times\mathcal P_g(Y)\to\R \cup \{+\infty\}$ be  measurable. We assume that $C$(resp. $\tilde C$) is convex and lower semicontinuous in the second argument, and such that there exists a constant $K>0$ which satisfies for all $(x,p)\in X\times\mathcal P_g(Y)$
	\begin{equation}\label{CgrowthexponentrWOT1}
	C(x,p)\ge -K\left(f(x)+\int_Yg(y)\,p(dy)\right) \quad\left(\text{resp.}\quad \tilde C(x,p)\ge -K\left(f(x)+\int_Yg(y)\,p(dy)\right)\right).
      \end{equation}
      For $\mu\in\mathcal P_f(X)$ and $\nu\in\mathcal P_g(Y)$ (resp. with $\mu\le_c \nu$), there exists $\pi^*\in\Pi(\mu,\nu)$ which minimises $V_C(\mu,\nu)$ (resp. $V^M_C(\mu,\nu)$) and  if $C$ is strictly convex in the second argument and $V_C(\mu,\nu)$ (resp. $V^M_C(\mu,\nu)$) is finite, then $V_C(\mu,\nu)$ (resp. $V^M_C(\mu,\nu)$) admits a unique minimiser.
      
	For $k\in\N$, let $\mu^k\in\mathcal P_f(X)$ and $\nu^k\in\mathcal P_g(Y)$ (resp. with $\mu^k\le_c\nu^k$) converge in $\mathcal P_f(X)$ and $\mathcal P_g(Y)$ as $k\to+\infty$ to $\mu$ and $\nu$, respectively.
	Suppose moreover that one of the following holds true:
	\begin{enumerate}[label = (\Alph*)]
		\item\label{it:assumption A lsc} $C$ (resp. $\tilde C$) is lower semicontinuous in both arguments.
		\item\label{it:assumption B lsc} $\mu^k$ converges strongly to $\mu$ as $k\to+\infty$.
	\end{enumerate}
	Then \begin{equation}\label{stabilityVM}
		V_C(\mu,\nu)\le\liminf_{k\to+\infty}V_C(\mu^k,\nu^k)\quad\left(\text{resp.}\quad V^M_C(\mu,\nu)\le\liminf_{k\to+\infty}V^M_C(\mu^k,\nu^k) \right).
		\end{equation}
	\end{theorem}
\begin{remark}\label{remctildec}
  {\color{black} If $X$ and $Y$ are Polish subspaces of $\R$ and $C:X\times\mathcal P_g(Y)\to\R\cup\{+\infty\}$ with $g \in \mathcal F^{|\cdot|}(Y)$ is convex in its second argument (resp. is lower semicontinuous in either its second argument or in both arguments, resp. satisfies \eqref{CgrowthexponentrWOT1}), then so does $\tilde C$. Indeed, as $Y\ni y\mapsto y\in\R$ belongs to $\Phi_g(Y)$,  $\{(x,p)\in X \times\mathcal P_g(Y) \colon x=\int_Y y \, p(dy)\}$ is a closed subset of $X\times\mathcal P_g(Y)$ and for fixed $x\in X$, $\{p\in\mathcal P_g(Y) \colon x=\int_Y y\, p(dy)\}$ is convex.}
\end{remark}


To conclude with stability of the WOT problem, we derive upper semicontinuity. 
This relies on the following (trivial) extension of Proposition 2.3 in our companion paper \cite{BeJoMaPa21a}. 
This extension is then an easy consequence of the equivalent definitions stated in Definition \ref{AWf+gdef}.

\begin{proposition}\label{prop:1b}
	Let $f \colon X \to [1,+\infty)$ and $g \colon Y \to [1,+\infty)$ be continuous.
	Let $\mu^k \in \mathcal P_f(X)$, $\nu^k\in\mathcal P_g(Y)$, $k\in\N$, respectively converge to $\mu$ and $\nu$ in $\mathcal P_f(X)$ and $\mathcal P_g(Y)$ respectively.
	Then there is for any $\pi\in\Pi(\mu,\nu)$ a sequence of couplings $\pi^k\in\Pi(\mu^k,\nu^k)$, $k\in\N$ converging to $\pi$ in $\mathcal{AW}_{f\oplus g}$.
\end{proposition}

In the martingale setting, we recall the main result of the companion paper \cite[Theorem 2.6]{BeJoMaPa21a}, namely that any martingale couplings whose marginals are approximated by probability measures in the convex order can be approximated by martingale couplings with respect to the adapted Wasserstein distance.
We invoke here a trivial extension of the aforementioned result which is again a direct consequence of the equivalent definitions stated in Definition \ref{AWf+gdef}. 

{\color{black}
\begin{theorem}\label{thm:1b}
	Let $X$ and $Y$ be Polish subspaces of $\R$, and let $f \in \mathcal F^{|\cdot|}(X)$ and $g \in \mathcal F^{|\cdot|}(Y)$.
	Let $\mu^k \in \mathcal P_f(X)$ and $\nu^k\in\mathcal P_g(Y)$, $k\in\N$, be in the convex order and respectively converge to $\mu$ and $\nu$ in $\mathcal P_f(X)$ and $\mathcal P_g(Y)$. 
	Let $\pi\in\Pi_M(\mu,\nu)$. 
	Then there exists a sequence of martingale couplings $\pi^k\in\Pi_M(\mu^k,\nu^k)$, $k\in\N$ converging to $\pi$ in $\mathcal{AW}_{f\oplus g}$.
\end{theorem}}

{\color{black}\begin{proof}[Proof of Theorem \ref{thm:1b}]
Let for $\eta\in\mathcal P(X)$ (resp. $\mathcal P(Y)$) $\tilde\eta$ denote the image of $\eta$ by the injection $\iota_X:X\ni x\mapsto x\in\R$ (resp. $\iota_Y:Y\ni y\mapsto y\in\R$). By continuity of the injections, the sequences $(\tilde \mu^k)_{k\in\N}$ and $(\tilde\nu^k)_{k\in\N}$ respectively converge weakly to $\tilde\mu$ and $\tilde\nu$. Since $\iota_X\in\Phi_f(X)$ and $\iota_Y\in\Phi_g(Y)$, $\int_\R|x|\tilde\mu^k(dx)=\int_X|\iota_X(x)|\mu^k(dx)\underset{k\to\infty}\longrightarrow\int_X|\iota_X(x)|\mu(dx)=\int_\R|x|\tilde\mu(dx)$ and $\int_\R|y|\tilde\nu^k(dy)\underset{k\to\infty}\longrightarrow\int_\R|y|\tilde\nu(dy)$ and the convergences hold in $\mathcal P_1(\R)$. Let $\pi\in\Pi_M(\mu,\nu)$ and $\tilde\pi$ denote the image of $\pi$ by the injection from $X\times Y$ to $\R^2$. Then $\tilde \pi\in\Pi_M(\tilde \mu,\tilde \nu)$ and, by \cite[Theorem 2.6]{BeJoMaPa21a}, there is a sequence $\tilde \pi^k \in \Pi_M(\tilde \mu^k,\tilde \nu^k)$, $k \in \N$ such that $J(\tilde \pi^k)$ converges to $J(\tilde \pi)$ in $\mathcal P_1(\R\times\mathcal P_1(\R))$ as $k\to+\infty$.

    The restrictions $\pi^k=\tilde\pi^k|_{\mathcal B(X\times Y)}$ and  $\pi=\tilde\pi|_{\mathcal B(X\times Y)}$ respectively belong to $\Pi_M(\mu^k,\nu^k)$ and $\Pi_M(\mu,\nu)$ and are such that $J(\tilde \pi^k)$ is the image of $J(\pi^k)$ by the mapping $\iota_{X,\mathcal P(Y)}:X\times \mathcal P(Y)\ni (x,p)\mapsto (x,\tilde p)\in\R\times \mathcal P(\R)$. The open ball  in $X\times{\mathcal P}(Y)$ centred at $(x,p)$ with radius $\varepsilon>0$ for the sum of the usual real distance on $X$ and the Prokhorov metric with the usual real distance on ${\mathcal P}(Y)$ is the preimage  by $\iota_{X,\mathcal P(Y)}$ of the open ball centred at $(x,\tilde p)$ with radius $\varepsilon$ for the sum of the usual distance on $\R$ and the Prokhorov metric on ${\mathcal P}(\R)$. Since the induced topologies on $X$ and $Y$ are metrizable by the usual distance on the real line, and, as $X$ is separable for the induced topology, the Prokhorov metric metrisizes the weak convergence topology in ${\mathcal P}(X)$, we deduce that any open subset of $X\times \mathcal P(Y)$ is the preimage of an open subset of $\R\times \mathcal P(\R)$ by $\iota_{X,\mathcal P(Y)}$. By the open subset characterization of the weak convergence in Portmanteau's theorem, we deduce that $J(\pi^k)$ converges weakly to $J(\pi)$ in ${\mathcal P}(X\times{\mathcal P}(Y))$. With the convergence of the marginals $\mu^k$ to $\mu$ in $\mathcal P_f(X)$ and $\nu^k$ to $\nu$ in $\mathcal P_g(Y)$, we conclude in view of Definition \ref{AWf+gdef} that $\pi^k$ converges to $\pi$ in $\mathcal{AW}_{f\oplus g}$ as $k\to\infty$.
\end{proof}}


\begin{theorem}[Upper semicontinuity]\label{thm:2upp} Let $X$ and $Y$ be Polish spaces {\color{black}(resp.\ $X$ and $Y$ be Polish subspaces of $\R$)}.
	Let $f:X\to[1,+\infty)$ and $g:Y\to[1,+\infty)$ be continuous {\color{black}(resp. $f\in \mathcal F^{|\cdot|}(X)$ and $g \in \mathcal F^{|\cdot|}(Y)$)}, $C:X\times\mathcal P_g(Y)\to\{-\infty\}\cup\R$ be measurable, upper semicontinuous in the second argument, and such that there exists a constant $K>0$ which satisfies for all $(x,p)\in X\times\mathcal P_g(Y)$
	\begin{equation}\label{CgrowthexponentrWOT2}
	C(x,p)\le K\left(f(x)+\int_Yg(y)\,p(dy)\right).
	\end{equation}
	For $k\in\N$, let $\mu^k\in\mathcal P_f(X)$ and $\nu^k\in\mathcal P_g(Y)$ (resp.\ with $\mu^k\le_c\nu^k$) converge in $\mathcal P_f(X)$ and $\mathcal P_g(Y)$ as $k\to+\infty$ to $\mu$ and $\nu$ respectively. 	
	Suppose moreover that one of the following holds true:
	\begin{enumerate}[label = (\Alph*')]
		\item\label{it:assumption A} $C$ is upper semicontinuous in both arguments.
		\item\label{it:assumption B} $\mu^k$ converges strongly to $\mu$ as $k\to+\infty$.
	\end{enumerate}
	Then
	\begin{equation}\label{uscVM}
	\limsup_{k\to+\infty}V_C(\mu^k,\nu^k)\le V_C(\mu,\nu)\quad\left(\text{resp.}\quad\limsup_{k\to+\infty}V^M_C(\mu^k,\nu^k)\le V^M_C(\mu,\nu)\right).
	\end{equation}
\end{theorem}
\begin{theorem}[Stability]\label{thm:2Polish2} Let $X$ and $Y$ be Polish spaces {\color{black}(resp.\ $X$ and $Y$ be Polish subspaces of $\R$)}.
	Let $f:X\to[1,+\infty)$ and $g:Y\to[1,+\infty)$ be continuous {\color{black}(resp. $f\in \mathcal F^{|\cdot|}(X)$ and $g \in \mathcal F^{|\cdot|}(Y)$)
        , $C:X\times\mathcal P_g(Y)\to\R$ be measurable, continuous in the second argument  and such that there exists a constant $K>0$ which satisfies for all $(x,p)\in X\times\mathcal P_g(Y)$
	\begin{equation}\label{CgrowthexponentrWOT2absolute}
	\vert C(x,p)\vert\le K\left(f(x)+\int_Yg(y)\,p(dy)\right).
      \end{equation}
      Also assume that $C$ (resp. $\tilde C$ defined in \eqref{deftildeC}) is convex in its second argument.}
	For $k\in\N$, let $\mu^k\in\mathcal P_f(X)$ and $\nu^k\in\mathcal P_g(Y)$ (resp.\ with $\mu^k\le_c\nu^k$) converge in $\mathcal P_f(X)$ and $\mathcal P_g(Y)$ as $k\to+\infty$ to $\mu$ and $\nu$ respectively. 	
	Suppose moreover that one of the following holds true:
	\begin{enumerate}[label = (\Alph*'')]
		\item\label{it:assumption Aa} $C$ is continuous in both arguments.
		\item\label{it:assumption Bb} $\mu^k$ converges strongly to $\mu$ as $k\to+\infty$.
	\end{enumerate}
	Then
	\begin{equation}\label{stabilityVM2}
	V_C(\mu^k,\nu^k)\underset{k\to+\infty}{\longrightarrow}V_C(\mu,\nu)\quad\left(\text{resp.}\quad V^M_C(\mu^k,\nu^k)\underset{k\to+\infty}{\longrightarrow}V^M_C(\mu,\nu)\right),
      \end{equation}
      For $k\in\N$ let $\pi^{k,*}\in\Pi(\mu^k,\nu^k)$ (resp. $\pi^{k,*}\in\Pi^M(\mu^k,\nu^k)$) be a minimiser of $V_C(\mu^k,\nu^k)$ (resp. $V^M_C(\mu^k,\nu^k)$), the existence of which is granted by Theorem \ref{thm:2Polish}. Then any accumulation point of $(\pi^{k,*})_{k\in\N}$ for the weak convergence topology is a minimiser of $V_C(\mu,\nu)$ (resp. $V^M_C(\mu,\nu)$). If the latter has a unique minimiser $\pi^*$, then
	\begin{equation}\label{eq:convergence WOT Polish}
	\pi^{k,*}\underset{k\to+\infty}{\longrightarrow}\pi^*\quad\text{in $\mathcal P_{f \oplus  g}(X \times Y)$ {\color{black}(resp. in $\mathcal P_{f \oplus  g}(X \times Y)$)}}.
	\end{equation}
	
	If $C$ is moreover strictly convex in the second argument, then the convergence \eqref{eq:convergence WOT Polish} holds in $\mathcal{AW}_{f\oplus g}$.
\end{theorem}


\section{Proofs}
\label{sec:proofs}
\subsection{Proofs of the stability theorems}\label{sec:stability}
This section is devoted to the proof of Theorem \ref{thm:2Polish}, Theorem \ref{thm:2upp} and Theorem \ref{thm:2Polish2} about the stability of \eqref{WOT2} and \eqref{WMOT2}.

\begin{proof}[Proof of Theorem \ref{thm:2Polish}] 

\textcolor{black}{Since in the martingale case, for $\mu\le_c\nu$, $V_C^M(\mu,\nu)=V_{\tilde C}(\mu,\nu)$, the statements concerning \eqref{WMOT2} follow from those concerning \eqref{WOT2} that we now prove.}

	Let $(\pi^n)_{n\in\N}\in \Pi(\mu,\nu)^\N$ be such that $\int_X C(x,\pi^n_x)\,\mu(dx)$ converges to $V_C(\mu,\nu)$ as $n\to+\infty$. By tightness of $\mu$ and $\nu$ we deduce the existence of a subsequence $(\pi^{n_l})_{l\in\N}$ of $(\pi^n)_{n\in\N}$ which converges to some $\pi^*\in \Pi(\mu,\nu)$ with respect to the weak convergence topology and therefore the topology of $\mathcal P_{f\oplus g}(X\times Y)$ since $\pi^{n_l}(f\oplus g)=\mu(f)+\nu(f)=\pi^*(f\oplus g)$ for all $l\in\N$. 
	\textcolor{black}{Since, for each $l\in\N$, the first marginal of $\pi^{n_l}$ is $\mu$, combining Lemma \ref{lem:convStable} and Proposition \ref{prop: C Pf lsc} \ref{it:Cpix lsc stable} below we then have}
	\[
	V_C(\mu,\nu)\le\int_X C(x,\pi^*_x)\,\mu(dx)\le\liminf_{l\to+\infty}\int_X C(x,\pi^{n_l}_x)\,\mu(dx)= V_C(\mu,\nu),
	\]
	which shows that $\pi^*$ is a minimiser for $V_C(\mu,\nu)$. The finiteness of $V_C(\mu,\nu)$ and the strict convexity of $C(x,\cdot)$ for all $x\in X$ yield uniqueness of the minimisers. Indeed when $\pi,\tilde\pi\in \Pi(\mu,\nu)$ then $\frac{1}{2}(\pi+\tilde\pi)\in\Pi(\mu,\nu)$. When, moreover, $\pi\neq\tilde\pi$, then $\mu(\{x\in X\colon\pi_x\neq \tilde \pi_x\})>0$ and since $C(x,\frac{1}{2}(\pi_x+\tilde \pi_x))\le \frac{1}{2}(C(x,\pi_x)+C(x,\tilde\pi_x))$ with strict inequality when $\pi_x\neq\tilde\pi_x$,
	\begin{equation}\label{eq:strict convexity cost function}
	\int_X C\left(x,\frac{\pi_x+\tilde \pi_x}{2}\right)\, \mu(dx)<\frac{1}{2}\left(\int_X C(x,\pi_x)\, \mu(dx)+\int_X C(x,\tilde \pi_x)\, \mu(dx)\right).
	\end{equation}
	
	We now show \eqref{stabilityVM}. Let $(V_C(\mu^{k_\ell},\nu^{k_\ell}))_{l\in\N}$ be a subsequence of $(V_C(\mu^k,\nu^k))_{k\in\N}$ such that 
	\[ \lim_{l \to +\infty} V_C(\mu^{k_\ell},\nu^{k_\ell}) = \liminf_{k\to+\infty}V_C(\mu^k,\nu^k), \]
	and denote, for $l \in \N$, an optimizer of $V_C(\mu^{k_\ell},\nu^{k_\ell})$ by $\pi^{k_\ell,\ast}$.
	Let $\tilde\pi\in\Pi(\mu,\nu)$ be an accumulation point of $(\pi^{k_\ell,*})_{l\in\N}$ in $\mathcal P_{f \oplus g}(X \times Y)$, which exists by relative compactness of the marginals. 
	Then, \textcolor{black}{by Proposition \ref{prop: C Pf lsc} \ref{it:Cpix lsc}  below under \ref{it:assumption A lsc} and by Lemma \ref{lem:convStable} and Proposition \ref{prop: C Pf lsc} \ref{it:Cpix lsc stable} below under \ref{it:assumption B lsc}}, we find that
	\begin{equation*}
	\liminf_{k\to+\infty} V_C(\mu^k,\nu^k)=\lim_{l\to+\infty}\int_X C(x,\pi^{k_\ell,*}_x)\,\mu^{k_\ell}(dx)\ge\int_X C(x,\tilde\pi_x)\,\mu(dx)\ge  V_C(\mu,\nu).
	\end{equation*}
\end{proof}
\begin{proof}[Proof of Theorem \ref{thm:2upp}]
    Let $\hat\Pi(\mu,\nu)=\Pi(\mu,\nu)$ and $\hat V_C(\mu,\nu)=V_C(\mu,\nu)$ in the classic setting, and $\hat\Pi(\mu,\nu)=\Pi_M(\mu,\nu)$ and $\hat V_C(\mu,\nu)=V^M_C(\mu,\nu)$ in the martingale setting. Let $\pi\in\hat\Pi(\mu,\nu)$. By Proposition \ref{prop:1b} in the classic setting and Theorem \ref{thm:1b} in the martingale setting, there exists a sequence $\pi^k\in\hat\Pi(\mu^k,\nu^k)$, $k\in\N$, which converges to $\pi$ in $\mathcal{AW}_{f\oplus g}$, which is equivalent to $J(\pi^k)$ converging to $J(\pi)$ in $\mathcal P_{f \oplus \hat g}(X \times \mathcal P(Y) )$.
	
\textcolor{black}{We then have by Lemma \ref{lem:caratheodorylimits} \ref{it:convlsc} under Assumption \ref{it:assumption A} and by  Lemma \ref{lem:convStable} and Lemma \ref{lem:caratheodorylimits} \ref{it:convlowCaratheodory} under \ref{it:assumption B} (Lemma \ref{lem:caratheodorylimits} is applied with $c$, $Y$ and $g$ replaced by $-C$, $\mathcal P_g(Y)$ and $\hat g$) that}
	\begin{equation}\label{importantConvergence}
	\limsup_{k\to\infty}\int_{X\times\mathcal P_g(Y)}C(x,p)\,J(\pi^k)(dx,dp)\le\int_{X\times\mathcal P_g(Y)}C(x,p)\,J(\pi)(dx,dp).
	\end{equation}
        Since $$\hat V_C(\mu^k,\nu^k)\le	\int_{X}C(x,\pi^k_x)\,\mu^k(dx)=	\int_{X\times\mathcal P_g(Y)}C(x,p)\,J(\pi^k)(dx,dp),$$
        we deduce that
        $$\limsup_{k\to\infty}\hat V_C(\mu^k,\nu^k)\le \int_{X}C(x,\pi_x)\,\mu(dx)$$
        and conclude by taking the infimum of the right-hand side with respect to $\pi\in\hat\Pi(\mu,\nu)$.
\end{proof}
\begin{proof}[Proof of Theorem \ref{thm:2Polish2}]
We adopt the same notation as in the proof of Theorem \ref{thm:2upp}. Combining Theorems \ref{thm:2Polish} and \ref{thm:2upp} and \textcolor{black}{Remark \ref{remctildec} (to check the lower semicontinuity of $\tilde C$ in the martingale case)} yields the stability result \eqref{stabilityVM2}. For $k\in\N$, let $\pi^{k,*}\in\hat\Pi(\mu^k,\nu^k)$ be a minimiser of $\hat V_C(\mu^k,\nu^k)$. For any subsequence $(\pi^{k_\ell,*})_{\ell\in\N}$ of $(\pi^{k,*})_{k\in\N}$ converging weakly to some $\tilde\pi$, \textcolor{black}{the convergence also holds in $\mathcal P_{f \oplus g}(X \times Y)$ and we have by Proposition \ref{prop: C Pf lsc} \ref{it:Cpix lsc} below under Assumption \ref{it:assumption Aa}  and by Lemma \ref{lem:convStable} and Proposition \ref{prop: C Pf lsc} \ref{it:Cpix lsc stable} below  under \ref{it:assumption Bb} that}
	\[
	\hat V_C(\mu,\nu)=\lim_{\ell\to+\infty}\hat V_C(\mu^{k_\ell},\nu^{k_\ell})=\lim_{\ell\to+\infty}\int_X C(x,\pi^{k_\ell,*}_x)\,\mu(dx)\ge\int_X C(x,\tilde\pi_x)\,\mu(dx)\ge\hat V_C(\mu,\nu),
	\]
	so $\tilde\pi$ is a minimiser of $\hat V_C(\mu,\nu)$. In particular if $\hat V_C(\mu,\nu)$ has a unique minimiser $\pi^*$, it is the unique accumulation point with respect to the weak convergence topology of the tight sequence $(\pi^{k,*})_{k\in\N}$, which therefore converges to $\pi^*$ weakly and even in $\mathcal P_{f \oplus g}(X \times Y)$ since its marginals converge in $\mathcal P_f(X)$ and $\mathcal P_g(Y)$ respectively.
	
	Let us now assume that $C$ is moreover strictly convex in its second argument. By Theorem \ref{thm:2Polish}, this implies uniqueness of the minimisers.	Let $\pi^*$ be the only minimiser of $\hat V_C(\mu,\nu)$. To conclude the proof, it is enough to show that $J(\pi^{k,*})$ converges to $J(\pi^*)$ in $\mathcal P_{f \oplus \hat g}(X \times \mathcal P(Y))$ as $k$ goes to $+\infty$. Let $P^*\in\mathcal P_{f \oplus \hat g}(X\times\mathcal P(Y))$ be an accumulation point of $(J(\pi^{k,*}))_{k\in\N}$, which exists by Lemma \ref{lem:rel comp product space}.
	
	To conclude, 
	it suffices to show that $P^*=J(\pi^*)$, which is achieved in three steps. Let $\hat\Lambda(\mu,\nu)=\Lambda(\mu,\nu)$ (see the definition \eqref{defLambda} below) in the classic setting and $\hat\Lambda(\mu,\nu)=\Lambda_M(\mu,\nu)$ (see the definition \eqref{defLambdaM} below) in the martingale setting. First we show that
	\begin{equation}\label{eq:P in biglambda}
	P^*\in\hat\Lambda(\mu,\nu).
	\end{equation}
	
	Next, we show that $J(\pi^*)$ and $P^*$ both minimise
	\[
	\tilde V_C(\mu,\nu):=\inf_{P\in\hat\Lambda(\mu,\nu)}\int_{X\times\mathcal P_g(Y)}C(x,p)\,P(dx,dp).
	\]
	
	Finally, we show the uniqueness of minimisers of $\tilde V_C(\mu,\nu)$.
	
	Let $(J(\pi^{k_\ell,*}))_{\ell\in\N}$ be a subsequence converging to $P^*$ in $\mathcal P_{f \oplus \hat g}(X \times \mathcal P(Y))$. By Lemma \ref{lem: continuity I} below we have
	\[
	\int_{(x,p)\in X\times\mathcal P_g(Y)}p(dy)\,J(\pi^{k_\ell,*})(dx,dp)\underset{\ell\to+\infty}{\longrightarrow}\int_{(x,p)\in X\times\mathcal P_g(Y)}p(dy)\,P^*(dx,dp),
	\]
	where the convergence holds in \textcolor{black}{$\mathcal P_{g}(Y)$} as $\ell$ goes to $+\infty$. Since the left-hand side is $\nu^{k_\ell}$, which converges to $\nu$ in $\mathcal P_g(Y)$ and therefore in the weak topology, we deduce by uniqueness of the limit that the right-hand side is $\nu$, hence $P^*\in\Lambda(\mu,\nu)$.
	In the martingale setting, since, as $f,g\in\mathcal F^{|\cdot|}(\R)$, $X\times\mathcal P_g(Y)\ni(x,p)\mapsto\left\vert x-\int_Y y\,p(dy)\right\vert\in\Phi_{f \oplus \hat g}(X\times\mathcal P_g(Y))$, we have that
	\[
	0=\int_{X\times\mathcal P_g(Y)}\left\vert x-\int_Y y\,p(dy)\right\vert\,J(\pi^{k_\ell,*})(dx,dp)\underset{\ell\to+\infty}{\longrightarrow}\int_{X\times\mathcal P_g(Y)}\left\vert x-\int_Y y\,p(dy)\right\vert\,P^*(dx,dp),
	\]
	hence $P^*\in\Lambda_M(\mu,\nu)$.
	
	Let us show that $J(\pi^*)$ and $P^*$ both minimise $\tilde V_C(\mu,\nu)$. Note that since $P^*\in\hat\Lambda(\mu,\nu)$, we have $P^*(X\times\mathcal P_g(Y))=1$. Since $(J(\pi^{k_\ell,*}))_{\ell\in\N}$ converges to $P^*$ in $\mathcal P_{f \oplus \hat g}(X \times \mathcal P(Y))$, we find by Lemma \ref{lem:caratheodorylimits}
	\begin{equation}\label{importantConvergence2}
	\int_{X\times\mathcal P_g(Y)}C(x,p)\,J(\pi^{k_\ell,*})(dx,dp)\underset{\ell\to+\infty}{\longrightarrow}\int_{X\times\mathcal P_g(Y)}C(x,p)\,P^*(dx,dp).
	\end{equation}
	
	Then \eqref{importantConvergence2} and the definition of $\pi^{k_\ell,*}$ yield
	\begin{align}\begin{split}\label{PJpiopt}
	\int_{X\times\mathcal P_g(Y)}C(x,p)\,P^*(dx,dp)&=\lim_{\ell\to+\infty}\int_{X\times\mathcal P_g(Y)} C(x,p)\,J(\pi^{k_\ell,*})(dx,dp)\\
	&=\lim_{\ell\to+\infty}\hat V_C(\mu^{k_\ell},\nu^{k_\ell})\\
	&=\hat V_C(\mu,\nu)\\
	&=\int_{X\times\mathcal P_g(Y)}C(x,p)\,J(\pi^*)(dx,dp).
\end{split}	\end{align}
	
	Let $P(dx,dp)=\mu(dx)\,P_x(dp)\in\hat\Lambda(\mu,\nu)$. 
	Then $\mu(dx)\,\int_{p\in\mathcal P_g(Y)}p(dy)\,P_x(dp)\in\hat\Pi(\mu,\nu)$, and by applying Proposition \ref{infinite-dimensional Jensen inequality} below in the last inequality, we find
	\begin{align}\begin{split}\label{intCJconvergesintCP}
	\int_{X\times\mathcal P_g(Y)}C(x,p)\,J(\pi^*)(dx,dp)&=\int_X C(x,\pi^*_x)\,\mu(dx)\\
	&=\hat V_C(\mu,\nu)\\
	&\le\int_X C\left(x,\int_{p\in\mathcal P_g(Y)}p(dy)\,P_x(dp)\right)\,\mu(dx)\\
	&\le\int_X\int_{\mathcal P_g(Y)}C(x,p)\,P_x(dp)\,\mu(dx),
	\end{split}
	\end{align}
	\textcolor{black}{which proves $\hat V_C(\mu,\nu)=\tilde V_C(\mu,\nu)$ with the latter attained by $J(\pi^*)$ and, according to \eqref{PJpiopt}, by $P^*$.}
	
	\textcolor{black}{Let us finally prove that $J(\pi^*)$ is the only minimiser of $\tilde V_C(\mu,\nu)$, which implies that $P^*=J(\pi^*)$. Since for all $\pi \in \hat \Pi(\mu,\nu)$
	\[	\int_X C(x,\pi_x) \, \mu(dx) = \int_{X \times \mathcal P_g(Y)} C(x,p) \, J(\pi)(dx,dp),\]
$J(\pi^*)$ is the only minimizer of $\tilde V_C(\mu,\nu)$ in the image of $\hat \Pi(\mu,\nu)$ by $J$. 
        To conclude, it is enough to check that any minimiser $\tilde P\in\hat\Lambda(\mu,\nu) $ actually belongs to this image. For $x\in X$, let $\tilde\pi_x(dy)=\int_{p\in\mathcal P_g(Y)}p(dy)\,\tilde P_x(dp)$ and $\tilde\pi(dx,dy)=\mu(dx)\,\tilde\pi_x(dy)$. Then $J(\tilde\pi)\in\hat\Lambda(\mu,\nu)$ and Proposition \ref{infinite-dimensional Jensen inequality} below yields
	\[
	\int_{X\times\mathcal P_g(Y)}C(x,p)\,J(\tilde\pi)(dx,dp)=\int_X C(x,\tilde\pi_x)\,\mu(dx)\le\int_X\int_{\mathcal P_g(Y)}C(x,p)\,\tilde P_x(dp)\,\mu(dx).\]
By optimality of $\tilde P$, this inequality is an equality, hence for $\mu(dx)$-almost every $x\in X$ we have
	\[
	C(x,\tilde\pi_x)=\int_{\mathcal P_g(Y)}C(x,p)\,\tilde P_x(dp),
	\]
	and therefore $\tilde P_x=\delta_{\tilde\pi_x}$ by the equality case of Proposition \ref{infinite-dimensional Jensen inequality} below, or equivalently $\tilde P=J(\tilde\pi)$.}
	
	
\end{proof}

      \subsection{Proof of the stability of the dual value function of VIX-futures}\label{sec:prvix} 
\begin{proof}[Proof of Theorem~\ref{thm:VIX}]
  \color{black}{
  The open interval $(0,+\infty)$ is a Polish subspace of $\R$ and the function $f(x)=|\ln x|+|x|$ belongs to ${\cal F}^{|\cdot|}((0,+\infty))$ (the fact that $f$ is bounded from below by $1$ follows from the inequality $|\ln x|\ge 1-x$ valid for $x\in(0,1]$).
  The cost function
  $$(0,+\infty)\times \mathcal P_f((0,+\infty))\ni(x,p)\mapsto C(x,p)=-\sqrt{\frac{2}{T_2-T_1}\left(\int_{(0,+\infty)}\ln\left(\frac x y\right)p(dy)\right)\vee 0}$$
   is continuous and such that
$$|C(x,p)|\le\frac 1 2+\frac{1}{T_2-T_1}\int_{(0,+\infty)}\ln\left(\frac x y\right)p(dy)\le \left(\frac 1 4+\frac{1}{T_2-T_1}\right)(f(x)+p(f)),$$
and that the modified cost function $\tilde C$ defined in \eqref{deftildeC} is convex in its second argument by concavity of the square root function.
    Since for each couple $(\mu,\nu)\in\mathcal P_f((0,+\infty))$ with $\mu\le_c \nu$,
    $$-D_{\textrm{super}}(\mu,\nu)=V^M_C(\mu,\nu),$$
    the conclusion follows from Theorem \ref{thm:2Polish2} applied with $X=Y=(0,+\infty)$ and $g=f$.}
 \end{proof}

\subsection{Proof of the stability of the stretched Brownian motion}\label{sec:prstr}
The stability of the unidimensional stretched Brownian motion in $\mathcal{W}_r$-topology stated in Corollary \ref{cor:stretched bm} actually holds in the stronger $\mathcal{AW}_r$-topology.
\begin{proposition}[Stability of the unidimensional stretched Brownian motion]\label{prop:stretched bm} Under the assumptions of Corollary \ref{cor:stretched bm},\[
	\mathcal{AW}_r^r\left(\mathcal L(M^k_0,(M^k_t)_{t\in[0,1]}),\mathcal L(M^*_0,(M^*_t)_{t\in[0,1]})\right)\le\left(\frac{r}{r-1}\right)^r\mathcal{AW}_r^r\left(\mathcal L(M^k_0,M^k_1),\mathcal L(M^*_0,M^*_1)\right),
	\]
	and the right-hand side vanishes as $k$ goes to $+\infty$.
      \end{proposition}The proof of the proposition relies on the following Lemma.

\begin{lemma}\label{lem:stability stretched bm}
	Let $\rho>1$, and $C_\rho:\R\times\mathcal P^\rho(\R)\to\R$ be defined for all $(x,p)\in\R\times\mathcal P^\rho(\R)$ by $C_\rho(x,p)=\mathcal W_\rho^\rho(p,\gamma)$, where $\gamma\in\mathcal P^\rho(\R)$ does not weight points. Let $V^M_{C_\rho}$ be the value function given by \eqref{WMOT2} for the cost function $C_\rho$.
	
	Let $r\ge\rho$ and $\mu^k,\nu^k\in\mathcal P^r(\R)$, $k\in\N$ be in convex order and converge respectively to $\mu$ and $\nu$ in $\mathcal W_r$. Then $\lim_{k\to+\infty}V^M_{C_\rho}(\mu^k,\nu^k)=V^M_{C_\rho}(\mu,\nu)$ and the optimisers are converging in $\mathcal{AW}_r$.
\end{lemma}
\begin{proof}
	By \textcolor{black}{Theorem \ref{thm:2Polish2} applied with $X=Y=\R$, $f(x)=1+|x|^r$ and $g(y)=1+|y|^r$}, it is sufficient to show that $p \mapsto \mathcal W_\rho^\rho(\gamma,p)$ is strictly convex. Since $\gamma$ does not weight points, the unique $\mathcal W_\rho$-optimal coupling between $\gamma$ and $p\in{\cal P}_\rho(\R)$ is the comonotonous coupling  $\chi^p$ given by the map $x\mapsto F_p^{-1}(F_\gamma(x))$ i.e. the image of $\gamma$ by $x\mapsto (x,F_p^{-1}(F_\gamma(x)))$. For $q\in{\cal P}_\rho(\R)$ and $\lambda\in (0,1)$ the coupling $\chi=(1-\lambda)\chi^p+\lambda\chi^q$ between $\gamma$ and $(1-\lambda)p+\lambda q$ is not given by a map unless $F_q^{-1}(u)=F_p^{-1}(u)$ for all $u\in (0,1)$ i.e. $p=q$. Therefore, when $p\neq q$,
	\[ (1 - \lambda) \mathcal W_\rho^\rho(\gamma,p) + \lambda \mathcal W_\rho^\rho(\gamma,q) = \int |x-y|^\rho \, \chi(dx,dy) > \mathcal W_\rho^\rho( \gamma,(1 - \lambda) p + \lambda q). \]
\end{proof}\begin{proof}[Proof of Proposition \ref{prop:stretched bm}]Let $\gamma=\mathcal N(0,1)$ be the unidimensional standard normal distribution and $C_2:\R\times\mathcal P^2(\R)\to\R$ be defined for all $(x,p)\in\R\times\mathcal P^2(\R)$ by $C_2(x,p)=\mathcal W_2^2(p,\gamma)$. Let $V^M_{C_2}$ be the value function given by \eqref{WMOT2} for the cost function $C_2$.
	
	In the setting of Corollary \ref{cor:stretched bm}, let $\pi^*\in\Pi_M(\mu,\nu)$, resp. $\pi^k\in\Pi_M(\mu^k,\nu^k)$ be optimal for $V^M_{C_2}(\mu,\nu)$, resp. $V^M_{C_2}(\mu^k,\nu^k)$. For $(x,b)\in\R\times\R^{[0,1]}$, let $B=(B_t)_{t\in[0,1]}$ be a Brownian motion and
	\[
	G^k(x,b)=\left(\E\left[F_{\pi^k_x}^{-1}(F_\gamma(B_1-B_t+b_t))\right]\right)_{t\in[0,1]}\text{ and }G^*(x,b)=\left(\E\left[F_{\pi^*_x}^{-1}(F_\gamma(B_1-B_t+b_t))\right]\right)_{t\in[0,1]}.
	\]
	
	According to \eqref{eq:stretched bm solution}, $(M^k_0,(M^k_t)_{t\in[0,1]})$ and $(M^*_0,(M^*_t)_{t\in[0,1]})$ are respectively distributed according to
	\[
	\eta^k(dx,df):=\mu^k(dx)\,(G^k(x,\cdot)_\ast W)(df)\quad\text{and}\quad \eta^*(dx,df):=\mu(dx)\,(G^*(x,\cdot)_\ast W)(df),
	\]
	where $W$ denotes the Wiener measure on ${\mathcal C}([0,1])$. Let $\chi^k\in\Pi(\mu^k,\mu)$ be optimal for $\mathcal{AW}_r(\pi^k,\pi)$. Then
	\begin{align*}
	\mathcal{AW}_r^r(\eta^k,\eta^*)&\le\int_{\R\times\R}\left(\vert x-x'\vert^r+\mathcal W_r^r(G^k(x,\cdot)_\ast W,G(x',\cdot)_\ast W)\right)\,\chi^k(dx,dx').
	\end{align*}
	
	According to \eqref{eq:stretched bm solution}, for $\mu^k(dx)$-almost every $x\in\R$, $G^k(x,B)$ is the stretched Brownian motion from $\delta_x$ to $\pi^k_x$, hence it is a continuous $(\mathcal F_t)_{t\in[0,1]}$-martingale, where $(\mathcal F_t)_{t\in[0,1]}$ is the natural filtration associated to $B$. Similarly, for $\mu(dx')$-almost every $x\in\R$, $G^*(x',B)$ is a continuous $(\mathcal F_t)_{t\in[0,1]}$-martingale. Therefore, for $\chi^k(dx,dx')$-almost every $(x,x')\in\R\times\R$, $G^k(x,B)-G^*(x',B)$ is a continuous $(\mathcal F_t)_{t\in[0,1]}$-martingale. Using Doob's martingale inequality for the second inequality, the fact that $F_\gamma(B_1)$ is uniformly distributed on $(0,1)$ for the first equality and the fact that the comonotonous coupling between $\pi^k_x$ and $\pi^*_{x'}$ is optimal for $\mathcal W_r(\pi^k_x,\pi^*_{x'})$ for the second equality, we get for $\chi^k(dx,dx')$-almost every $(x,x')\in\R\times\R$
	\begin{align*}
	\mathcal W_r^r(G^k(x,\cdot)_\ast W,G(x',\cdot)_\ast W)&\le\E\left[\sup_{t\in[0,1]}\left\vert G^k(x,B)_t-G^*(x',B)_t\right\vert^r\right]\\
	&\le\left(\frac{r}{r-1}\right)^r\E[\vert G^k(x,B)_1-G^*(x',B)_1\vert^r]\\
	&=\left(\frac{r}{r-1}\right)^r\E[\vert F_{\pi^k_x}^{-1}(F_\gamma(B_1))-F_{\pi^*_{x'}}^{-1}(F_\gamma(B_1)\vert^r]\\
	&=\left(\frac{r}{r-1}\right)^r\mathcal W_r^r(\pi^k_x,\pi^*_{x'}).
	\end{align*}
	
	We deduce that
	\begin{align*}
	\mathcal{AW}_r^r(\eta^k,\eta^*)&\le\left(\frac{r}{r-1}\right)^r\int_{\R\times\R}\left(\vert x-x'\vert^r+\mathcal W_r^r(\pi^k_x,\pi^*_{x'})\right)\,\chi^k(dx,dx')=\left(\frac{r}{r-1}\right)^r\mathcal{AW}_r^r(\pi^k,\pi^*),
	\end{align*}
	where the right-hand side vanishes as $k$ goes to $+\infty$ by virtue of Lemma \ref{lem:stability stretched bm}.
\end{proof}

\subsection{Proof of sufficiency of martingale $C$-monotonicity}
\label{sec:MartingaleMonotonicity}

In this section we prove the claim that martingale $C$-monotonicity is sufficient for optimality for \eqref{WMOT2}. Theorem \ref{thm:3r} is a special case of the next statement for $f(x)=1+|x|^r$ and $g(y)=1+|y|^r$.
{\color{black}
\begin{theorem}[Sufficiency] \label{thm:3}
	Let $f,g \in{\cal F}^{|\cdot|}(\R)$, $\mu\in\mathcal P_f(\R)$, $\nu \in \mathcal P_g(\R)$ be in convex order, and 
	$C\colon \R \times \mathcal P_g(\R) \to \R$ be a measurable cost function,
	continuous in the second argument and such that there exists a constant $K>0$ which satisfies
	\[
	\forall(x,p)\in \R \times\mathcal P_g(\R),\quad C(x,p)\le K\left(f(x) +\int_\R g(y) \,p(dy)\right),
	\]
Let $\Gamma$ be martingale $C$-monotone and $\pi \in \Pi_M(\mu,\nu)$ be such that we have \eqref{eq:martingale C monotone coupling}. 
Then $\pi$ is optimal for \eqref{WMOT2}.
\end{theorem}
For $g:X \to[1,+\infty)$ continuous, we denote 
\begin{equation}\label{defFg}
\mathcal F_g(X):=\{f:\R\to[1,+\infty)\text{ continuous }\colon\forall y\in X,\ f(y)\ge g(y)\}.
\end{equation}
and
\begin{equation}\label{defFg+}
\mathcal F_g^+(X):=\left\{f\in\mathcal F_g(X)\colon\exists h:\R_+\to[1,+\infty),\ \frac{h(t)}{t}\underset{t\to+\infty}{\longrightarrow}+\infty\text{ and }f=h\circ g\right\}.
\end{equation}}

\begin{proof}[Proof of Theorem \ref{thm:3}]
	{\color{black} 
	Let $h\in\mathcal F_g(\R)$ be such that $\nu(h)<+\infty$, whose purpose will be revealed later in the proof. 
	To demonstrate the main idea without further technical details, we assume for now that $\mu$ is concentrated on a Polish subspace $X$ of $\R$, and the restriction $C\vert_{X\times\mathcal P_h(\R)}$ is continuous.
	Moreover, we denote by $\tilde f$ the restriction of $f$ to $X$, thus $\tilde f \in \mathcal F^{|\cdot|}(X)$.
	Let $X_n:\Omega\to\R$, $n\in\N$ be independent random variables identically distributed according to $\mu$ and $\mathcal G\subset\Phi_{\tilde f \oplus \hat g}(X\times\mathcal P(\R))$ be a countable family which determines the convergence in $\mathcal P_{\tilde f \oplus \hat g}(X\times \mathcal P(\R))$ (see \cite[Theorem 4.5.(b)]{EtKu09}).
	By the law of large numbers, almost surely, for all $\psi\in\mathcal G \cup \{ C|_{X\otimes \mathcal P(\R)} \}$,
	\begin{align}\label{lawLargeNumbers}\begin{split}
	\frac1n\sum_{k=1}^n\int_{X\times\mathcal P(\R)}\psi(x,p)\,\delta_{(X_k,\pi_{X_k})}(dx,dp)&=\frac1n\sum_{k=1}^n\psi(X_k,\pi_{X_k})\\
	&\underset{n\to+\infty}{\longrightarrow}\E[\psi(X_1,\pi_{X_1})]\\
	&=\int_{X} \psi(x,\pi_x)\,\mu(dx)\\
	&=\int_{X\times\mathcal P(\R)}\psi(x,p)\,J(\pi)(dx,dp),
	\end{split}
	\end{align}
	
	Moreover, almost surely, for all $n\in\N$, 
	\begin{equation}\label{belongsMartingaleCMonotone}
	(X_n,\pi_{X_n})\in\Gamma\cap(X\times\mathcal P_g(\R)),
	\end{equation}
	and by the law of large numbers again, we have almost surely
	\begin{equation}\label{lawLargeNumbersnuf}
	\frac1n\sum_{k=1}^n\pi_{X_k}(h)\underset{n\to+\infty}{\longrightarrow}\E[\pi_{X_1}(h)]=\int_{X} \pi_x(h) \,\mu(dx)=\nu(h).
	\end{equation}
	
	Let then $\omega\in\Omega$ be such that \eqref{lawLargeNumbers}, \eqref{belongsMartingaleCMonotone} and \eqref{lawLargeNumbersnuf} hold when evaluated at $\omega$ and set $x_n=X_n(\omega)$ and $\pi^n(dx,dy)=\frac1n\sum_{k=1}^n\delta_{x_k}(dx)\,\pi_{x_k}(dy)$ for $n\in\N$. Then $\pi^n$ has first marginal $\mu^n=\frac1n\sum_{k=1}^n\delta_{x_k}$ and second marginal $\nu^n=\int_{x\in X}\pi_x(dy)\,\mu^n(dx)$.
	We deduce that $\pi^n$ is a martingale coupling between $\mu^n$ and $\nu^n$ such that
	\[
		J(\pi^n)=\frac1n\sum_{k=1}^n\delta_{(x_k,\pi_{x_k})}\underset{n\to+\infty}{\longrightarrow}J(\pi)\quad \text{in }\mathcal P_{\tilde f \oplus \hat g}(X\times \mathcal P(\R)).
	\]
	In particular we have convergence of the marginals in $\mathcal P_{\tilde f}(X)$ and $\mathcal P_g(\R)$ respectively.
	The second marginals are even converging in $\mathcal P_h(\R)$ since $\nu^n(h)$ converges to $\nu(h)$ as $n\to+\infty$ by \eqref{lawLargeNumbersnuf} evaluated at $\omega$.
	Thus, the convergence of $(J(\pi^n))_{n \in \N}$ to $J(\pi)$ even holds in $\mathcal P_{\tilde f \oplus \hat h}(X\times\mathcal P(\R))$ (see Definition \ref{AWf+gdef}) and therefore $\mathcal P_{\tilde f \oplus \hat h}(X\times\mathcal P_h(\R))$ by Lemma \ref{lem:equivalent topologies} \ref{it:equal topologies Pf+g} below.

	Since $(x,\pi^n_x)\in \Gamma$ for $\mu^n(dx)$-almost every $x$, we have according to Remark \ref{rk:finiteSupportOptimal} that
	\begin{equation*}
	V^M_C(\mu^n,\nu^n)=\int_{X} C(x,\pi^n_x) \, \mu^n(dx),
	\end{equation*}
	where we recall that the value function $V^M_C$ is defined in \eqref{WMOT2}.
	In that context, $C|_{X\times \mathcal P_h(\R)}$ is a continuous function on $X\times \mathcal P_h(\R)$ that is dominated from above by a positive multiple of $\tilde f \oplus \hat g$.
	We apply Theorem \ref{thm:2upp} to get
	\begin{align*}
	\int_{X} C(x,\pi_x)\,\mu(dx)&=\int_{X\times\mathcal P_h(\R)}C(x,p)\,J(\pi)(dx,dp)\\
	&=\lim_{n\to+\infty}\int_{X\times\mathcal P_h(\R)}C(x,p)\,J(\pi^n)(dx,dp)\\
	&=\lim_{n\to+\infty}\int_X C(x,\pi^n_x)\,\mu^n(dx)\\
	&=\lim_{n\to+\infty}V^M_C(\mu^n,\nu^n)\\
	&\le V^M_C(\mu,\nu),
	\end{align*}
	where we used \eqref{lawLargeNumbers} for the second equality.
	Hence, $\pi$ is optimal for $V^M_C(\mu,\nu)$.}
	
	Next, we drop the additional joint-continuity assumption on $C$. 
	Since $\nu(g)<+\infty$, there exists by the de La Vallée Poussin theorem $h\in\mathcal F_g^+(\R)$ such that $\nu(h)<+\infty$.
	For $N\in\N^*$, let $B_N = \{ p \in \mathcal P_g(\R) \colon p(h) \leq N \}$, which is a compact subset of $\mathcal P_g(\R)$ by Lemma \ref{lem: Wfball Wrcompact} below, and $\mathcal C(B_N)$ be the set of continuous functions from $B_N$ to $\R$, endowed with the topology of uniform convergence. The map $\phi^N \colon \R \to \mathcal C(B_N)$ given by $\phi^N(x) = C(x,\cdot)|_{B_N}$ is Borel measurable due to \cite[Theorem 4.55]{Ch06}.
	Let $\varepsilon\in(0,1)$. By Lusin's theorem there is for every $N\in\N^*$ a compact set $K^N\subset \R$ such that the restriction $\phi^N\vert_{K^N}$ is continuous and $\mu(K^N) \geq 1 - \frac{\epsilon}{2^N}$.
	We have
	\[  \mu\left( \bigcap_{N\in\N^*} K^N \right) \geq 1 - \sum_{N\in\N^*} \mu\left( (K^N)^c \right) \geq 1 - \sum_{N\in\N^*} \frac{\epsilon}{2^N} = 1 - \epsilon. \]
	
	Let $K^\epsilon = \bigcap_{N\in\N^*} K^N$, then for all $N\in\N^*$ the restriction $\phi^N\vert_{K^\varepsilon}$ is continuous.
	We claim that $C\vert_{K^\epsilon \times \mathcal P_h(\R)}$ is continuous w.r.t.\ the product topology of $\R \times \mathcal P_h(\R)$.
	To this end, take any sequence $(x_k,p_k)_{k \in \N} \in (K^\epsilon \times \mathcal P_h(\R))^\N$ with limit point $(x,p) \in K^\epsilon \times \mathcal P_h(\R)$.
	Since $p_k \to p$ in $\mathcal P_h(\R)$ as $k$ goes to $+\infty$, the sequence $(p_k(h))_{k\in\N}$ is convergent to $p(h)$ and therefore bounded so there exists $N \in \N$ such that $p,p_k \in B_N$ for all $k \in \N$.
	As $\phi^N(x_k)$ converges uniformly to $\phi^N(x)$, we have
	\[  C(x_k,p_k) = \phi^N(x_k)(p_k) \underset{k\to+\infty}{\longrightarrow} \phi^N(x)(p) = C(x,p). \]
	
	Therefore, $C\vert_{K^\epsilon \times \mathcal P_h(\R)}$ is continuous.
	
	Let $\mu^\varepsilon=\frac{1}{\mu(K^\varepsilon)}\mu\vert_{K^\varepsilon}$, $\pi^\varepsilon=\mu^\varepsilon\times\pi_x=\frac{1}{\mu(K^\varepsilon)}\pi\vert_{K^\epsilon\times\R}$ and $\nu^\varepsilon$ be the second marginal of $\pi^\varepsilon$. Obviously $\mu^\varepsilon$ is concentrated on $K^\varepsilon$. Since $\mu(K^\epsilon)\mu^\varepsilon\le\mu$ and $\pi^\varepsilon_x=\pi_x$, $\pi^\varepsilon$ is a martingale coupling and satisfies $(x,\pi^\varepsilon_x)\in\Gamma$ for $\mu^\varepsilon(dx)$-almost every $x$. Finally, $\mu(K^\epsilon)\nu^\varepsilon(h)= \int_{K^\epsilon}\pi_x(h)\,\mu(dx)\le\nu(h)<+\infty$, hence $\nu^\varepsilon\in\mathcal P_h(\R)$. Therefore the first part applied with $(K^\varepsilon,\mu^\varepsilon,\nu^\varepsilon,\pi^\varepsilon)$ replacing $(X,\mu,\nu,\pi)$ ensures that $\pi^\varepsilon$ is optimal for $V^M_C(\mu^\varepsilon,\nu^\varepsilon)$.
        
        \textcolor{black}{Since $C(x,p)$ is bounded from above by a positive multiple of $f(x)+\hat h(p)$ and $\mu(f)+\nu(h)<\infty$, either $\int_\R C(x,\pi_x)\mu(dx)=-\infty$ or $\int_\R |C(x,\pi_x)|\mu(dx)<\infty$. In the latter case, by Lebesgue's theorem, $\int_{K^\epsilon} C(x,\pi_x) \,\mu(dx)$ converges to $\int_{\R} C(x,\pi_x) \,\mu(dx)$ as $\epsilon\to 0$ and so does $$\frac{1}{\mu(K^\epsilon)}\int_{K^\epsilon} C(x,\pi_x) \,\mu(dx)=\int_{\R} C(x,\pi_x) \, \mu^\epsilon(dx)=V^M_C(\mu^\epsilon,\nu^\epsilon).$$ 
To synthesize the two cases, $\liminf_{\epsilon\to 0}V^M_C(\mu^\epsilon,\nu^\epsilon)\ge\int_\R C(x,\pi_x)\mu(dx)$. The marginals $(\mu^\epsilon)_{\epsilon > 0}$ converge to $\mu$ in $\mathcal P_f(\R)$ and strongly, whereas the marginals $(\nu^\epsilon)_{\epsilon > 0}$ converge to $\nu$
	in $\mathcal P_h(\R)$ for $\epsilon \searrow 0$. By Theorem \ref{thm:2upp}, we conclude that
	\[ V_C^M(\mu,\nu)\ge \limsup_{\epsilon\to 0} V_C^M(\mu^{\epsilon},\nu^{\epsilon})\ge \liminf_{\epsilon\to 0} V_C^M(\mu^{\epsilon},\nu^{\epsilon})\ge\int_{\R} C(x,\pi_x)\,\mu(dx),\]
	proving optimality of $\pi$.}
      \end{proof}
    In the same way, Corollary \ref{cor:finite optimalityr} is a special case of the next corollary for the choice $f(x)=1+|x|^r$ and $g(y)=1+|y|^r$.
\begin{corollary} \label{cor:finite optimality}
	Let $f,g \in{\cal F}^{|\cdot|}(\R)$, $\mu\in\mathcal P_f(\R)$ and $\nu \in \mathcal P_g(\R)$ be in convex order,
	$c\colon \R \times \R \to \R$ be measurable and such that $y\mapsto c(x,y)$ is continuous for all $x \in \R$ and $\sup_{(x,y)\in\R^2}\frac{|c(x,y)|}{f(x)+g(y)}<\infty$.	
	Then $\pi \in \Pi_M(\mu,\nu)$ is concentrated on a set finitely optimal for $c$ if and only if $\pi$ is optimal for \eqref{MOT}.
\end{corollary}
The proof of Corollary \ref{cor:finite optimality} relies on the next lemma.
\begin{lemma}\label{lem:approxsupfinwrho}
Let $g\in{\cal F}^{|\cdot|}(\R)$, $p\in\mathcal P_g(\R)$ and $A$ be a Borel subset of $\R$ such that $p(A)=1$. Then there exists a sequence of probability measures finitely supported in $A$ which converges to $p$ in ${\mathcal P}_g(\R)$.
 \end{lemma} 
\begin{proof}[Proof of Corollary \ref{cor:finite optimality}]
   The sufficient condition is stated in Lemma 1.11 \cite{BeJu16} under mere measurability of the cost function $c$. For the necessary condition let us suppose that $\pi$ is concentrated on a set $\tilde\Gamma$ finitely optimal for $c$ and define $\Gamma=\{(x,p)\in\R\times{\mathcal P}_g(\R):\int_\R\1_{\tilde \Gamma}(x,y)p(dy)=1\}$ and $\R\times{\mathcal P}_g(\R)\ni(x,p)\mapsto C(x,p)=\int_\R c(x,y)p(dy)$. The growth condition satisfied by $c$ and the continuity of this function in its second variable ensure that $C$ is continuous in its second argument and satisfies the growth assumption in Theorem \ref{thm:3}. 
   Since $1=\pi(\tilde\Gamma)=\int_\R\int_\R\1_{\tilde \Gamma}(x,y)\pi_x(dy)\mu(dx)$ and $\int_\R\int_\R g(y)\pi_x(dy)\mu(dx)=\int_\R g(y)\nu(dy)<\infty$, \eqref{eq:martingale C monotone coupling} holds.  Let $(x_1,p_1),\ldots,(x_N,p_N) \in \Gamma$, and $q_1,\ldots,q_N \in \mathcal P(\R)$ such that $\sum_{i=1}^N p_i = \sum_{i=1}^N q_i$, which implies $q_i\in{\mathcal P}_g(\R)$ for each $i\in\{1,\cdots,N\}$, and that $\int_\R y \, p_i(dy) = \int_\R y \, q_i(dy)$ for each $i\in\{1,\cdots,N\}$. 
   For each $i\in\{1,\cdots,N\}$, by Lemma \ref{lem:approxsupfinwrho}, there is a sequence $(p^k_i)_{k\in\N}$ of finitely supported probability measures such that $1=\int_\R\1_{\tilde \Gamma}(x_i,y)p^k_i(dy)$ converging to $p_i$ in ${\mathcal P}_g(\R)$, which, since $g\in{\cal F}^{|\cdot|}(\R)$, ensures that \textcolor{black}{$y_i^k:=\int_\R y p_i^k(dy)\underset{k\to\infty}{\longrightarrow}\int_\R y p_i(dy)=\int_\R y q_i(dy):=y_i$. Since $\frac 1 N\sum_{i=1}^N\delta_{y_i^k}$ (resp. $\frac{1}{N}\sum_{i=1}^Np_i^k$) converges to $\frac 1 N\sum_{i=1}^N\delta_{y_i}$ (resp. $\frac{1}{N}\sum_{i=1}^Np_i=\frac{1}{N}\sum_{i=1}^Nq_i$) in $\mathcal P_f(\R)$ (resp. in $\mathcal P_g(\R)$) as $k\to\infty$, by Theorem \ref{thm:1b} applied with $X=Y=\R$, there exists a sequence $\frac{1}{N}\sum_{i=1}^N \delta_{y_i^k}(dx)q_i^k(dy)$ of elements of $\Pi_M(\frac{1}{N}\sum_{i=1}^N \delta_{y_i^k},\frac{1}{N}\sum_{i=1}^Np_i^k)$ converging to the element $\frac{1}{N}\sum_{i=1}^N \delta_{y_i}(dx)q_i(dy)$ of $\Pi_M(\frac{1}{N}\sum_{i=1}^N \delta_{y_i},\frac{1}{N}\sum_{i=1}^Np_i)$ in $\mathcal{AW}_{f\oplus g}$. Up to replacing $(y_i,p_i,q_i,(p_i^k)_{k\in\N})$ by their image under a well chosen real translation depending on $i$ (one should then replace the resulting $q_i^k$ by their image by the inverse translation), we may assume that the $y_i$ are distinct so that the convergence in $\mathcal{AW}_{f\oplus g}$ implies that for each $i\in\{1,\cdots,N\}$, $q_i^k$ converges to $q_i$ in ${\mathcal P}_g(\R)$ as $k\to\infty$. Since $\frac{1}{N}\sum_{i=1}^N q^k_i=\frac{1}{N}\sum_{i=1}^N p^k_i$ and, by the martingale property, $\int_\R y q_i^k(dy)=y_i^k=\int_\R y p_i^k(dy)$, $\frac 1 N\sum_{i=1}^N\delta_{x_i}(dx)q^k_i(dy)$ is a competitor of $\frac 1 N\sum_{i=1}^N\delta_{x_i}(dx)p^k_i(dy)$ which is finitely supported on $\tilde \Gamma$. Hence  $$\sum_{i=1}^NC(x_i,p^k_i)=N\int_{\R\times\R}c(x,y)\frac 1 N\sum_{i=1}^N\delta_{x_i}(dx)p^k_i(dy)\le N\int_{\R\times\R}c(x,y)\frac 1 N\sum_{i=1}^N\delta_{x_i}(dx)q^k_i(dy)=\sum_{i=1}^NC(x_i,q^k_i).$$} Letting $k\to\infty$, we deduce by continuity of $C$ in the measure argument that $\sum_{i=1}^NC(x_i,p_i)\le \sum_{i=1}^NC(x_i,q_i)$. Therefore $\Gamma$ is martingale $C$-monotone and by Theorem \ref{thm:3}, $\pi$ is optimal for  \eqref{WMOT2} 
 and equivalently for \eqref{MOT} in view of the definition of $C$.

\end{proof}

\begin{proof}[Proof of Lemma \ref{lem:approxsupfinwrho}]
   For $U$ uniformly distributed on $(0,1)$ and $n\in\N^*$, let 
 $p^n_U:=\frac{1}{n}\sum_{j=1}^n\delta_{F_p^{-1}(\frac{j-U}{n})}$ 
 . For $\varphi:\R\to\R$ measurable and bounded, one has
 $$\E\left[\int_\R\varphi(y)p^n_U(dy)\right]=\frac 1 n\sum_{j=1}^n\E\left[\varphi\left(F_p^{-1}\left(\frac{j-U}{n}\right)\right)\right]=\frac 1 n\sum_{j=1}^nn\int_{\frac{j-1}{n}}^{\frac jn}\varphi(F_p^{-1}(v))dv=\int_0^1\varphi(F_p^{-1}(v))dv,$$
 with the last integral equal to $\int_\R\varphi(y)p(dy)$ by the inverse transform sampling. Hence the intensity of the random probability $p^n_U$ is equal to $p$. In particular, $\E[p^n_U(A)]=p(A)=1$ and \begin{equation}
   \E\left[\int_\R g(y)p^n_U(dy)\right]=\int_\R g(y)p(dy).\label{eq:espegal}
 \end{equation} Moreover, when $\varphi$ is continuous and bounded,

 \begin{align*}
   \int_\R \varphi(y)p^n_U(dy)=\int_0^1\varphi\left(F_p^{-1}\left(\frac{\lceil nv\rceil-U}{n}\right)\right)dv\rightarrow\int_0^1\varphi(F_p^{-1}(v))dv=\int_0^1\varphi(y)p(dy)\mbox{ as }n\to\infty,
 \end{align*}
 by Lebesgue's theorem, since the set of discontinuities of $F_p^{-1}$ is at most countable. Hence $p^n_U$ converges weakly to $p$ as $n\to\infty$ from which we deduce that $\liminf_{n\to\infty}\int_\R g(y)p^n_U(dy)\ge \int_\R g(y)p(dy)$. With \eqref{eq:espegal} and Fatou lemma, we deduce that a.s. $\liminf_{n\to\infty}\int_\R g(y)p^n_U(dy)=\int_\R g(y)p(dy)$ so that a.s. there is a subsequence of $(p^n_U)_n$ giving full weight to $A$ and converging to $p$ in ${\mathcal P}_g(\R)$.
 \end{proof} 

\appendix

\section{Appendix}
\label{sec:AppendixApplicationPaper}

We recall that the adapted weak topology can be defined as the initial topology under the embedding map $J$ from $\mathcal P(X\times Y)$ to $\mathcal P(X\times\mathcal P(Y))$, namely
\begin{equation}\label{defJ2}
J:\mathcal P(X\times Y)\ni\pi=\mu\otimes\pi_x\mapsto\mu(dx)\,\delta_{\pi_x}(dp)\in\mathcal P(X\times\mathcal P(Y)).
\end{equation}

Conversely, we can associate to a probability measure $P\in\mathcal P(\mathcal P(Y))$ its intensity $I(P)$
\begin{equation}
   I(P)(dy):=\int_{p\in\mathcal P(Y)}p(dy)\,P(dp)\in\mathcal P(Y).\label{def:int}
\end{equation} 
For the extended space $\mathcal P(X\times\mathcal P(Y))$ we naturally define the extended intensity $\hat I$ by
\begin{equation}\label{defExtendedIntensity}
\hat I:\mathcal P(X\times\mathcal P(Y))\ni P\mapsto\int_{p\in\mathcal P(Y)}\,p(dy)\,P(dx,dp)\in\mathcal P(X\times Y),
\end{equation}
which associates to each $P \in \mathcal P(X \times \mathcal P(Y))$ a coupling $\hat I(P) \in \mathcal P(X \times Y)$. We note that $\hat I$ is the left-inverse of $J$.

For $(\mu,\nu)\in\mathcal P(X)\times\mathcal P(Y)$, we define the set of extended couplings $\Lambda(\mu,\nu)$ between $\mu$ and $\nu$ as the set of probability measures on $\mathcal P(X\times\mathcal P(Y))$ whose extended intensity is a coupling between $\mu$ and $\nu$, that is
\begin{equation}\label{defLambda}
\Lambda(\mu,\nu)=\left\{P=\mu\otimes P_x\in\mathcal P(X\times\mathcal P(Y)) \colon \hat I(P) \in \Pi(\mu,\nu) \right\}.
\end{equation}

If $f \colon X \to \R^+$ and $g \colon Y \to \R^+$ are measurable functions, then any $P \in \Lambda(\mu,\nu)$ satisfies
\begin{equation}\label{moments P}
	\int_{X\times\mathcal P(Y)} f(x) \,P(dx,dp)= \mu(f), \quad\int_{X\times\mathcal P(Y)}\int_Yg(y)\,p(dy)\,P(dx,dp)=\nu(g).
\end{equation}

For $\mu,\nu\in\mathcal P^1(\R)$, the martingale counterpart $\Lambda_M(\mu,\nu)$ of $\Lambda(\mu,\nu)$ is given by the set of probabilities on $\mathcal P^1(\R\times\mathcal P^1(\R))$ satisfying
\begin{equation}\label{defLambdaM}
\Lambda_M(\mu,\nu)=\left\{P\in\Lambda(\mu,\nu)\colon\int_{\R} y\,p(dy)=x,\  P(dx,dp)\text{-a.s.}\right\}.
\end{equation}

\subsection{Extension from \texorpdfstring{$\mathcal P^r$}{Pr} to \texorpdfstring{$\mathcal P_f$}{Pf}.}
\label{sec:extensionPrPf}

We recall that unless explicitly stated otherwise, $\mathcal P(Y)$ is endowed with the weak convergence topology, and for any continuous map $f:Y\to[1,+\infty)$ we endow the space $\mathcal P_f(Y)=\{p\in\mathcal P(Y)\colon p(f)<+\infty\}$ with the topology induced by the following convergence: a sequence $(p_k)_{k\in\N}\in\mathcal P_f(Y)^\N$ converges in $\mathcal P_f(Y)$ to $p$ if and only if $p_k$ converges weakly to $p$ and $p_k(f)$ converges to $p(f)$ as $k\to+\infty$.

As mentioned in Section \ref{sec:main results2}, this extension emerged from the need to overcome the inconvenience of the non-compactness of the $\mathcal W_r$-balls $\{p\in\mathcal P^r(Y)\colon\mathcal W_r(p,\delta_{y_0})\le R\}$, $R>0$, for the $\mathcal W_r$-topology.
The following lemmas show that this extension enjoys flexibility as the usual Wasserstein topology and most importantly benefits of a helpful compactness result, see Lemma \ref{lem: Wfball Wrcompact} below.

\begin{remark} \label{rem:Pf}
	We continue with some remarks on the structure of $\mathcal P_f(Y)$:
	\begin{enumerate}[(1)]
		\item \label{it:rem Pf1}
		Convergence in $\mathcal P_f(Y)$ can be described differently: let $(p_k)_{k \in \N}$ converge
		to $p$ in $\mathcal P_f(Y)$, and let $g \in\mathcal C(Y)$ be such that $0 \leq g \leq f$.
                We have $p(g)\le\liminf_{k\to+\infty}p_k(g)$ and $p(f)-p(g)=p(f-g)\le\liminf_{k\to+\infty}p_k(f-g)=p(f)-\limsup_{k\to+\infty}p_k(g)$, hence $\limsup_{k\to+\infty}p_k(g)\le p(g)$. We deduce that
		\begin{equation} 
		\label{eq:Pf convergence}
		p_k\underset{k\to+\infty}{\longrightarrow}p \text{ in }\mathcal P_f(Y) \iff p_k(g)
		\underset{k\to+\infty}{\longrightarrow} p(g), \quad \forall g\in \Phi_f(Y),
		\end{equation}
		when $\Phi_f(Y) := \{g \in \mathcal C(Y) \colon g\text{ is absolutely dominated by a positive multiple of }f \}$.
		
		It is immediate that for $r\ge1$, this topology is finer than the one induced by $\mathcal W_r$ on $\mathcal P_f(Y)$ if $f$ 
              is  bounded from below by $y\mapsto1+d_Y^r(y,y_0)$.
		\item \label{it:rem Pf2}
		The set $\mathcal P_f(Y)$ is naturally embedded into the set $\mathcal M_+(Y)$ of all
		bounded positive Borel measures on $Y$, endowed with the weak topology, via the following continuous injection
		\[
		\iota \colon \mathcal P_f(Y) \to \mathcal M_+(Y),\quad \iota(p)(dy) = f(y)\, p(dy).
		\]
		
		Clearly, the topology on $\mathcal P_f(Y)$ coincides with the initial topology under $\iota$.
		Even more, the set $\iota(\mathcal P_f(Y)) = \{ m \in \mathcal M_+(Y) \colon m(\frac{1}{f}) = 1\}$
		is a closed subset of $\mathcal M_+(Y)$ since $\frac1f$ is continuous and bounded. As such, we deduce that $\mathcal P_f(Y)$ is a Polish
		space.  
		\item \label{it:rem Pf3}
		By \cite[Theorem 8.3.2 and the preceding discussion]{Bo00}, we have that the weak topology on
		$\mathcal M_+(Y)$ is induced by the norm
		\[ \lVert m_1 - m_2 \rVert_0 := \sup_{\substack{g \colon Y \to [-1,1] \\ g \text{ is 1-Lipschitz}} } (m_1(g) - m_2(g)). \]
		
		This permits us to define a metric on $\mathcal P_f(Y)$ via
		\begin{equation}\label{defWbarf2}
		\overline{ \mathcal W }_f(p,q) := \sup_{\substack{g \colon Y \to [-1,1], \\ g \text{ is
					$1$-Lipschitz}}} (p(fg) - q(fg)) = \lVert \iota(p) - \iota(q) \rVert_0.
		\end{equation}
		
		Thus, $\overline{ \mathcal W }_f$ is a complete metric compatible with the topology
		on $\mathcal P_f(Y)$.               
	\end{enumerate}
\end{remark}
From now on, we equip $\mathcal P_f(Y)$ with $\overline{\mathcal W}_f$.
A continuous function $f \colon Y \to [1,+\infty)$ can naturally be lifted to a continuous function
$\hat f \colon \mathcal P_f(Y) \to [1,+\infty)$ by setting
\begin{equation}
\label{eq:lifted f}
\hat f(p) := p(f).
\end{equation}
For any probability measure $P\in\mathcal P(\mathcal P(Y))$, we then have $P(\hat f)=I(P)(f)$ where the intensity $I(P)$ is defined in \eqref{def:int}.

For two maps $f:X\to\R$ and $g:Y\to\R$, we denote $f\oplus g:X\times Y\ni(x,y)\mapsto f(x)+g(y)$.

As we are solely interested in topological properties, the next lemma shows that we can freely switch between 
the spaces $\mathcal P_{\hat f}(\mathcal P(Y))$, $\mathcal P_{\hat f}(\mathcal P_f(Y))$, and
$\mathcal P^1(\mathcal P_f(Y))$, where the latter is defined as $\mathcal P^r(X)$ with $(1,\mathcal P_f(Y),\overline{\mathcal W}_f)$ replacing $(r,X,d_X)$.

\begin{lemma}\label{lem:equivalent topologies} 
	\begin{enumerate}[(a)]
		\item\label{it:equal topologies Pf} Let $f:Y\to[1,+\infty)$ be continuous. Then
		\begin{equation}\label{lem:equal topologies Pf}
		\mathcal P_{\hat f}(\mathcal P(Y))=\mathcal P_{\hat f}(\mathcal P_f(Y)),
		\end{equation}
		and their topologies coincide. 
		Moreover, if $\mathcal P_f(Y)$ is endowed with the metric $\overline{\mathcal W}_f$, see \eqref{defWbarf}, then
		\begin{equation}\label{lem:equal topologies Pf2}
		\mathcal P_{\hat f}(\mathcal P(Y))=\mathcal P_{\hat f}(\mathcal P_f(Y))=\mathcal P^1(\mathcal P_f(Y)),
		\end{equation}
		and their topologies are equal.
		\item\label{it:equal topologies Pf+g} Let $f:X\to[1,+\infty)$ and $g:Y\to[1,+\infty)$ be continuous. Then
		\begin{equation}\label{lem:equel topologies Pf+g}
		\mathcal P_{f\oplus\hat g}(X\times\mathcal P(Y))=\mathcal P_{f\oplus\hat g}(X\times\mathcal P_g(Y)),
		\end{equation}
		and their topologies coincide.
	\end{enumerate}
\end{lemma}
\begin{rk}
	The equalities \eqref{lem:equal topologies Pf}, \eqref{lem:equal topologies Pf2} and \eqref{lem:equel topologies Pf+g} are to be understood up to an identification:
	For two measurable sets $Z'\subset Z$, we say a probability measure $p\in\mathcal P(Z)$ also belongs to $\mathcal P(Z')$ if $p(Z')=1$.
	This is achieved by identifying $p\in\mathcal P(Z)$ with $p'\in\mathcal P(Z')$ where for any measurable subset $A\subset Z'$, $p'(A)=p(A\cap Z')$. 
\end{rk}
\begin{proof} Let us prove \ref{it:equal topologies Pf}. The inclusion $\mathcal P_{\hat f}(\mathcal P(Y))\supset\mathcal P_{\hat f}(\mathcal P_f(Y))$ is straightforward. Conversely, let $P\in\mathcal P_{\hat f}(\mathcal P(Y))$. Then by definition,
	\[
	P(\hat f)=\int_{\mathcal P(Y)}p(f)\,P(dp)<+\infty,
	\]
	which can only hold if $p(f)$ is $P(dp)$-almost everywhere finite, or equivalently $P(\mathcal P_f(Y))=1$, hence $\mathcal P_{\hat f}(\mathcal P(Y))\subset\mathcal P_{\hat f}(\mathcal P_f(Y))$ and therefore we have equality. To see that the two topologies match, let us show that
	\[
	P^k\underset{k\to+\infty}{\longrightarrow}P\ \text{in }\mathcal P_{\hat f}(\mathcal P_f(Y))\iff P^k\underset{k\to+\infty}{\longrightarrow}P\ \text{in }\mathcal P_{\hat f}(\mathcal P(Y)).
	\]
	
	Since the topology on $\mathcal P_f(Y)$ is finer than the weak topology on $\mathcal P(Y)$, we have $\mathcal C(\mathcal P(Y))\subset\mathcal C(\mathcal P_f(Y))$, so the direct implication is trivial. Conversely, suppose that $P^k$ converges in $\mathcal P_{\hat f}(\mathcal P(Y))$ to $P$ as $k$ goes to $+\infty$. Let $h\in\mathcal C(Y)$ be bounded. Then $\hat h\in\mathcal C(\mathcal P(Y))$ is bounded, and $I(P^k)(h)=P^k(\hat h)$ converges to $P(\hat h)=I(P)(h)$ as $k$ goes to $+\infty$. Moreover $I(P^k)(f)=P^k(\hat f)$ converges to $P(\hat f)=I(P)(f)$. This shows that $(I(P^k))_{k\in\N}$ converges in $\mathcal P_f(Y)$ to $I(P)$. Therefore $\{I(P^k)\colon k\in\N\}$ is relatively compact in $\mathcal P_f(Y)$. We deduce by Lemma \ref{lem:rel comp intensity} below that $\{P^k\colon k\in\N\}$ is relatively compact in $\mathcal P_{\hat f}(\mathcal P_f(Y))$. Let $Q$ be an accumulation point of $(P^k)_{k\in\N}$ in $\mathcal P_{\hat f}(\mathcal P_f(Y))$. In particular $Q$ is by the direct implication shown above an accumulation point of $(P^k)_{k\in\N}$ in $\mathcal P_{\hat f}(\mathcal P(Y))$, hence $Q=P$ by uniqueness of the limit since the topology is metrisable and therefore Hausdorff.
	
	Let us now prove the second part of \ref{it:equal topologies Pf}. We endow $\mathcal P_f(Y)$ with the metric $\overline{\mathcal W}_f$. To see that the sets $\mathcal P_{\hat f}(\mathcal P_f(Y))$ and $\mathcal P^1(\mathcal P_f(Y))$ are the
	same, we find
	\[ P(\hat f) < + \infty \iff \int_{\mathcal P(Y)}p(f)\,P(dp)<+\infty \iff \int_{\mathcal P(Y)} \overline{\mathcal W}_f(p,\delta_{y_0}) \, P(dp) < +\infty, \]
	which is an easy consequence of
	\[
	\forall p\in\mathcal P_f(Y),\quad  p(f) - f(y_0) \leq \overline{\mathcal W}_f(p,\delta_{y_0}) \leq p(f) + f(y_0).   \]
	
	Let us now prove \ref{it:equal topologies Pf+g}. We derive the equality $\mathcal P_{f\oplus\hat g}(X\times\mathcal P(Y))=\mathcal P_{f\oplus\hat g}(X\times\mathcal P_g(Y))$ as in \ref{it:equal topologies Pf} since
	\[
	P(f\oplus\hat g)=\int_{X\times\mathcal P(Y)}(f(x)+p(g))\,P(dx,dp)<+\infty,
	\]
	which can only hold if the second marginal of $P$ is concentrated on $\mathcal P_g(Y)$. To see that the topologies are equal, the only nontrivial part is, as in \ref{it:equal topologies Pf}, to show that if $(P^k)_{k\in\N}$ converges in $\mathcal P_{f\oplus\hat g}(X\times\mathcal P(Y))$, then $\{P^k\colon k\in\N\}$ is relatively compact in $\mathcal P_{f\oplus\hat g}(X\times\mathcal P_g(Y))$. Let then $(P^k)_{k\in\N}$ converge in $\mathcal P_{f\oplus\hat g}(X\times\mathcal P(Y))$ to some $P$. Recall moreover the definition of the extended intensity $\hat I$ given by \eqref{defExtendedIntensity}. \textcolor{black}{Let $h:X\times Y\to\R$ be continuous and bounded. Then the map $H:X\times\mathcal P(Y)\ni(x,p)\mapsto\int_Yh(x,y)\,p(dy)=\int_{X\times Y}h(z,y)\delta_x(dz)p(dy)$ is continuous and bounded. We deduce that $\hat I(P^k)(h)=P^k(H)$ converges to $P(H)=\hat I(P)(h)$ as $k$ goes to $+\infty$. Hence $(\hat I(P^k))_{k\in\N}$ converges weakly to $\hat I(P)$.} 
        Then by continuity of the projections the first marginal $\mu^k$, resp. the second marginal $\nu^k$ of $\hat I(P^k)$ converges weakly to the first marginal $\mu$, resp. the second marginal $\nu$ of $\hat I(P)$. Since the maps $f\oplus\hat0:(x,p)\mapsto f(x)$ and $0\oplus\hat g:(x,p)\mapsto\hat g(p)$ belong to $\mathcal C(X\times\mathcal P(Y))$ and are dominated by $f\oplus\hat g$, we also have that
	\[
	\mu^k(f)=P^k(f\oplus\hat0)\underset{k\to+\infty}{\longrightarrow}P(f\oplus\hat0)=\mu(f)\quad\text{and}\quad\nu^k(g)=P^k(0\oplus g)\underset{k\to+\infty}{\longrightarrow}P(0\oplus g)=\nu(g),
	\]
	which shows that $(\mu^k,\nu^k)_{k\in\N}$ converges in $\mathcal P_f(X)\times\mathcal P_g(Y)$ to $(\mu,\nu)$. Therefore $(\hat I(P^k))_{k\in\N}$ is tight in $\mathcal P(X\times Y)$ and both projections $\{\mu^k\colon k\in\N\}$ and $\{\nu^k\colon k\in\N\}$ are relatively compact respectively in $\mathcal P_f(X)$ and $\mathcal P_{\textcolor{black}{g}}(Y)$, so by Lemma \ref{lem:rel comp product space} below $\{P^k\colon k\in\N\}$ is relatively compact in $\mathcal P_{f\oplus\hat g}(X\times\mathcal P_g(Y))$, which proves the claim.
\end{proof}

\begin{lemma} \label{lem:rel comp intensity}
	A set $\mathcal A \subset \mathcal P_{\hat f}(\mathcal P_f(Y))$ is relatively compact if and only if the set of its intensities $I(\mathcal A) \subset \mathcal P_f(Y)$ is relatively compact.
\end{lemma}

\begin{proof}
	The first implication follows as in \cite[Lemma 2.4]{BaBePa18} by continuity of $I$, c.f.\ Lemma \ref{lem: continuity I} below.
	The reverse implication can be shown by pursuing the same idea as in \cite[Lemma 2.4]{BaBePa18} with slight modification: instead of considering the map $y \mapsto d_Y(y,y')^t$ we use $y \mapsto f(y)$.
\end{proof}

\begin{lemma} \label{lem:rel comp characterization}
	A set $\mathcal A \subset \mathcal P_f(Y)$ is relatively compact if and only if it is tight and
	\[	\forall \varepsilon > 0,\quad\exists R > 0,\quad\sup_{\mu \in \mathcal A} \int_{\{y\in Y\colon f(y) \geq R\}} f(y) \, \mu(dy) < \varepsilon.	\]
\end{lemma}

\begin{proof}
	The proof of this lemma runs along the lines of \cite[Lemma 2.5]{BaBePa18} when replacing $y \mapsto d_Y(y,y')^t$ by $y \mapsto f(y)$.
\end{proof}

\begin{lemma}\label{lem: Wfball Wrcompact} 
        Let $g \colon \R \to [1,+\infty)$ be continuous and \textcolor{black}{such that $\lim_{|x|\to\infty}g(x)=+\infty$.} Then for all $f$ in the set ${\mathcal F}_g^+(\R)$ defined in \eqref{defFg+}, the set $B_R:=\{p\in\mathcal P(\R)\colon p(f)\le R\}$ is a compact subset of $\mathcal P_g(\R)$.
\end{lemma}
\begin{proof} Let $R\ge0$, $(p_n)_{n\in\N}$ be a sequence in $B_R^\N$ and $\epsilon > 0$.
	There exists $r>0$ such that for all $x\in\R$, $|x| \ge r$ implies $f(x)\ge\frac R\varepsilon$. Let $K=\{x\in\R\colon\vert x\vert\le r\}$. 
	For all $n\in\N$, we have $R\ge p_n(f)\ge p_n(\R\backslash K)\frac R\varepsilon$,
	hence $p_n(\R\backslash K)\le\varepsilon$. So $(p_n)_{n\in\N}$ is tight, and by 
	Prokhorov's theorem there exists a subsequence, still denoted $(p_n)_{n\in\N}$ for notational
	simplicity, which converges weakly to $p\in\mathcal P(\R)$. Since $f$ is continuous and
	nonnegative, we have
	\[
	p(f)\le\liminf_{n\to+\infty}p_n(f)\le R,
	\]
	so $p(f)\in B_R$. It remains to show that this convergence also holds in $\mathcal P_g(\R)$. By Skorokhod's representation theorem, there exists for all $n\in\N$ a random variable $Z_n\sim p_n$, such that $(Z_n)_{n \in \N}$ converges almost surely to a random variable $Z \sim p$. For all $n\in\N$ we have
	\[
	p_n(g) = \E[g(Z_n)] \leq \E[f(Z_n)] = p_n(f)\le R,
	\]
	so by the de La Vallée Poussin theorem, $(g(Z_n))_{n\in\N}$ is uniformly integrable.
	We deduce that
	\[
	\lim_{n \to +\infty} p_n(g) = p(g)
	\]
	and $(p_n)_{n \in \N}$ converges in $\mathcal P_g(\R)$ to $p$,
	so $B_R$ is compact.
\end{proof}

For a probability measure $\pi\in\mathcal P(X\times Y)$, we denote by $\proj_X(\pi)$ and $\proj_Y(\pi)$ its $X$-marginal and $Y$-marginal, respectively. Recall moreover the definition of the extended intensity $\hat I$ given by \eqref{defExtendedIntensity}.
\begin{lemma} \label{lem:rel comp product space}
	Let $f \colon X \to [1,+\infty)$ and $g \colon Y \to [1,+\infty)$ be continuous.
	The following are equivalent:
	\begin{enumerate}[label =(\alph*)]
		\item A set $\Pi \subset \mathcal P(X \times Y)$ is tight and both projections,
		$\proj_X(\Pi) \subset \mathcal P_f(X)$ and $\proj_Y(\Pi)\subset \mathcal P_g(Y)$,
		are relatively compact.
		\item $J(\Pi)$ as a subset of $\mathcal P_{f \oplus \hat g}(X \times \mathcal P_g(Y))$
		is relatively compact.
	\end{enumerate}
	
	Conversely, the following are equivalent:
	\begin{enumerate}[label =(\alph*')]
		\item $\Lambda \subset \mathcal P_{f \oplus \hat g}(X \times \mathcal P_g(Y))$ is relatively
		compact.
		\item $\hat I(\Lambda) \subset \mathcal P(X \times Y)$ is tight, and both projections, 
		$\proj_X(\hat I(\Lambda)) \subset \mathcal P_f(X)$ and $\proj_Y(\hat I(\Lambda)) \subset 
		\mathcal P_g(Y)$, are relatively compact.
	\end{enumerate}
\end{lemma}
\begin{proof}
	For this lemma works the same proof as in \cite[Lemma 2.6]{BaBePa18} when using Lemma \ref{lem:rel comp intensity}, the characterisation of relative compactness given in Lemma \ref{lem:rel comp characterization} and continuity of $\hat I$, see Lemma \ref{lem: continuity I}.
\end{proof}

\begin{lemma}\label{lem: continuity I}
	Let $f \colon X \to [1,+\infty)$ and $g \colon Y \to [1, +\infty)$ be continuous.
	The maps
	\begin{align}
	\label{def:I Pf}
	&I \colon \mathcal P_{\hat g}(\mathcal P(Y)) \to \mathcal P_g(Y),\quad
	I(P)(dy) := \int_{\mathcal P(Y)} p(dy)\,P(dp),
	\\  \label{def:hatI Pf}
	&\hat I \colon \mathcal P_{f \oplus \hat g}(X \times \mathcal P(Y)) \to P_{f \oplus g}(X \times Y),\quad \hat I(P)(dx,dy) := \int_{\mathcal P(Y)}p(dy) \, P(dx,dp),
	\end{align}
	are continuous.
\end{lemma}
\begin{proof}
	Let $(P^k)_{k \in \N}$ be a sequence in $\mathcal P_{\hat g}(\mathcal P(Y))$
	with limit point $P$. Let $h:Y\to\R$ be continuous and bounded, then $\hat h:{\mathcal P(Y)}\to\R$ is continuous and bounded.
	Thus,
	\begin{align*}
	\lim\limits_{k\to+\infty} I(P^k)(h)= \lim\limits_{k \to +\infty} P^k(\hat h) = P(\hat h) = I(P)(h),
	\\  \lim\limits_{k\to+\infty} I(P^k)(g)=\lim\limits_{k\to+\infty}P^k(\hat g)=P(\hat g) = I(P)(g),
	\end{align*}
	which shows continuity of $I$.
	
	Next, let $(P^k)_{k \in \N}$ be a sequence in $\mathcal P_{f \oplus \hat g}(X \times \mathcal P(Y))$
	converging to $P$. \textcolor{black}{Let $h:X\times Y\to\R$ be continuous and bounded. Then the map $H:X\times\mathcal P(Y)\ni(x,p)\mapsto\int_Yh(x,y)\,p(dy)=\int_{X\times Y}h(z,y)\delta_x(dz)p(dy)$ is continuous and bounded.} Again, we find
	\begin{align*}
	\lim\limits_{k\to+\infty} \hat I(P^k)(h)= \lim\limits_{k \to +\infty} P^k(H) = P(H) = \hat I(P)(h),
	\\  \lim\limits_{k\to+\infty} \hat I(P^k)(f \oplus g)=\lim\limits_{k\to+\infty}P^k(f \oplus \hat g)=P(f \oplus \hat g) = \hat I(P)(f \oplus g),
	\end{align*}
	whereby we derive continuity of $\hat I$.
\end{proof}

\begin{proposition}\label{infinite-dimensional Jensen inequality} 
	Let $g \colon Y \to [1,+\infty)$ be continuous, $C:\mathcal P_g(Y)\to\R$ be convex, lower
	semicontinuous and lower bounded by a negative multiple of $\hat g$. 
	Then for all $Q\in\mathcal P_{\hat g}(\mathcal P(Y))$ holds
	\begin{equation}\label{eq:convex lsc jensen inequality}
	C\left(I(Q)\right)\le\int_{\mathcal P_g(Y)} C(p)\,Q(dp).
	\end{equation}
	
	If moreover $C$ is strictly convex, then \eqref{eq:convex lsc jensen inequality} is an equality if and only if
	$Q = \delta_{I(Q)}$.
\end{proposition}

\begin{proof} Let $Q\in\mathcal P_{\hat g}(\mathcal P(Y))$, $P_n:\Omega\to\mathcal P(Y)$, $n\in\N^*$ be independent random variables identically distributed according to $Q$ and $\mathcal G\subset\Phi_{\hat g}(\mathcal P(Y))$ be a countable family which determines the convergence in $\mathcal P_{\hat g}(\mathcal P(Y))$ (see \cite[Theorem 4.5.(b)]{EtKu09}). By the law of large numbers, almost surely, for all $\psi\in\mathcal G$,
	\begin{equation}\label{lawLargeNumbers2}
	\frac1n\sum_{k=1}^n\psi(P_k)\underset{n\to+\infty}{\longrightarrow}\E[\psi(P_1)]=Q(\psi)\quad\text{and}\quad\frac1n\sum_{k=1}^nC(P_k)\underset{n\to+\infty}{\longrightarrow}\E[C(P_1)]=Q(C).
	\end{equation}
	
	Let $\omega\in\Omega$ be such that \eqref{lawLargeNumbers2} holds when evaluated at $\omega$ and set $p_n=P_n(\omega)$ for $n\in\N^*$. Then $\left(\frac1n\sum_{k=1}^n\delta_{p_k}\right)_{n\in\N}$ converges in $\mathcal P_{\hat g}(\mathcal P(Y))$ to $Q$. By Lemma \ref{lem: continuity I}, $\frac1n\sum_{k=1}^np_k$ converges to $I(Q)$ \textcolor{black}{in $\mathcal P_g(Y)$} as $n\to+\infty$. By lower semicontinuity of $C$ for the first inequality, convexity of $C$ for the second one and \eqref{lawLargeNumbers2} evaluated at $\omega$ for the equality, we get
	\begin{equation}
		C(I(Q))\le\liminf_{n\to+\infty}C\left(\frac1n\sum_{k=1}^np_k\right)\le\liminf_{n\to+\infty}\frac1n\sum_{k=1}^nC(p_k)=Q(C).
	\end{equation}

	If $Q = \delta_{I(Q)}$ we have trivially equality in \eqref{eq:convex lsc jensen inequality}.
	So, assume that $Q$ is not concentrated on a single point, and that $C$ is strictly convex.
	There are $h \in \Phi_g(Y)$ and $b \in \R$ such that $A = \{ p \in \mathcal P_g(Y) \colon p(h) 
	\leq b \}$ satisfies
	\begin{equation}
	\label{eq:jensen separating}
	Q(A) > 0 \text{ and }Q(A^c) > 0.
	\end{equation}
	
	Indeed, pick any points $p_1,p_2 \in \mathcal P_g(Y)$, $p_1 \neq p_2$ in the support of $Q$,
	then the Hahn-Banach separation theorem provides $h \in \Phi_g(Y)$ and $b \in \R$ such that
	$p_1(h) < b < p_2(h)$. As both points lie in the support of $Q$, and $\{p \in \mathcal P_g(Y) \colon p(h) < b\}$ and $\{p \in \mathcal P_g(Y) \colon p(h) > b\}$ are open subsets containing
	$p_1$ and $p_2$, respectively, we obtain \eqref{eq:jensen separating}.
	Write $Q_1(dp) := \1_{A} \frac{Q(dp)}{Q(A)}$ and $Q_2(dp) := \1_{A^c} \frac{Q(dp)}{Q(A^c)}$.
	By the definition of $A$, we have that $I(Q_1)(h)\le b < I(Q_2)(h)$ and especially $I(Q_1) \neq I(Q_2)$.
	By \eqref{eq:convex lsc jensen inequality} we find
	\[
	\int_{\mathcal P_g(y)} C(p) \, Q_1(dp) \geq C(I(Q_1)) \text{ and }\int_{\mathcal P_g(Y)} C(p) \, Q_2(dp) \geq C(I(Q_2)).
	\]
	
	Hence, as $Q = Q(A) Q_1 + (1 - Q(A)) Q_2$ we get
	\begin{align*}
	\int_{\mathcal P_g(Y)} C(p) \, Q(dp) &= \int_{\mathcal P_g(Y)} C(p) Q(A) \, Q_1(dp) + \int_{\mathcal P_g(Y)} C(p) Q(A^c) \, Q_2(dp)
	\\  &\geq Q(A) C(I(Q_1)) + (1 - Q(A)) C(I(Q_2))
	\\  &> C(Q(A)I(Q_1) + (1 - Q(A))I(Q_2)) = C(I(Q)),
	\end{align*}
	where we used $I(Q_1) \neq I(Q_2)$ and strict convexity for the last inequality.
\end{proof}

\subsection{A Pormanteau-like theorem for Carathéodory maps}

Let $(\pi^k)_{k\in\N}$ be a sequence of probability measures defined on $X\times Y$ converging in 
$\mathcal P_{f \oplus g}(X \times Y)$ to $\pi$, and $c:X\times Y\to\R$ be a (lower) Carathéodory map, that is a measurable function which is (lower semi-) continuous in its second argument. 
The goal of the present section is to connect the asymptotic behaviour of $\int_{X\times Y}c(x,y)\,\pi^k(dx,dy)$ and $\int_{X\times Y}c(x,y)\,\pi(dx,dy)$ under relaxed assumption on the cost $c$.
We recall that $(\pi^k)_{k\in\N}$ is said to converge stably to $\pi$ if and only if for every bounded measurable map $g:X\to\R$ and bounded continuous map $h:Y\to\R$
\begin{equation}\label{defStableConvergence}
\int_{X\times Y}g(x)h(y)\,\pi^k(dx,dy)\underset{k\to+\infty}{\longrightarrow}\int_{X\times Y}g(x)h(y)\,\pi(dx,dy).
\end{equation}

We say that a sequence $(\mu^k)_{k \in \N}$ of probability measures on $\mathcal P(X)$ $K$-converges in total variation to $\mu$ if and only if for every subsequence $(\mu^{k_i})_{i \in \N}$ we have
\begin{equation*}
\frac{1}{n} \sum_{i = 1}^n \mu^{k_i} \underset{n\to+\infty}{\longrightarrow}\mu\quad \text{in total variation}.
\end{equation*}

\begin{lemma}\label{lem:convStable}
	Let $\pi, \pi^k \in \mathcal P(X \times Y)$, $k\in\N$ be with respective first marginal $\mu$, $\mu^k$.	All of the following statements are equivalent:
	\begin{enumerate}[label = (\alph*)]
		\item \label{it:stable convergence} $(\pi^k)_{k\in\N}$ converges to $\pi$ stably.
		\item \label{it:strong weak} $(\pi^k)_{k\in\N}$ converges to $\pi$ weakly and $(\mu^k)_{k \in \N}$ converges strongly to $\mu$.
		\item \label{it:KTV weak} $(\pi^k)_{k\in\N}$ converges to $\pi$ weakly and every subsequence of
		$(\mu^k)_{k \in \N}$ has an in total variation $K$-convergent sub-subsequence with limit $\mu$.
	\end{enumerate}
\end{lemma}

\begin{proof}
	We prove ``\ref{it:stable convergence}$\implies$\ref{it:strong weak}''.
	The definition of stable convergence given by \eqref{defStableConvergence} is in the Polish set-up by \cite[Theorem 8.10.65 (ii)]{Bo92} equivalent to 
	\[\int_{X\times Y} c(x,y) \, \pi^k(dx,dy) \underset{k\to+\infty}{\longrightarrow} \int_{X\times Y} c(x,y) \, \pi(dx,dy)\]
	for all $c\colon X \times Y \to \R$ which are bounded and Carath\'eodory.
	Thus, stable convergence is stronger than weak convergence.
	For all measurable subsets $A\subset X$, we find by setting $g=\1_A$ and $h=1$ in \eqref{defStableConvergence} that
	\[\mu^k(A)\underset{k\to+\infty}{\longrightarrow}\mu(A).\]
	
	Next we show ``\ref{it:strong weak}$\implies$\ref{it:KTV weak}''.
	
	Let $\mu^k(dx) = \rho^k(x) \, \mu(dx) + \eta^k(dx)$ be the Lebesgue decomposition of $\mu^k$ w.r.t.\ $\mu$.
	Since $\eta^k$ is singular to $\mu$ there is $N^k \in \mathcal B(X)$ such that $\eta^k(N^k) = \eta^k(X)$ and $\mu(N^k) = 0$.
	Define $N = \bigcup_{k \in \N} N^k \in \mathcal B(X)$, then $\eta^k(N) = \eta^k(X)$ for all $k\in \N$ and $\mu(N)$ vanishes as a countable union of null sets.
	Thus, $\eta^k(X) = \mu^k(N) \to \mu(N) = 0$ as $k \to +\infty$.
	Since $(\rho^k)_{k \in \N}$ is bounded in $L^1(\mu)$ there is by Koml\'os theorem a $K$-convergent subsequence to some limiting function $\rho \in L^1(\mu)$.
	We have
	\[	\frac{1}{n} \sum_{l = 1}^n \rho^{k_l} \underset{n\to+\infty}{\longrightarrow} \rho,\quad \mu\text{-a.s.}	\]
	
	By \cite[Corollary 4.5.7]{Bo00} the above convergence even holds in $L^1(\mu)$.
	We find for any measurable subset $A\subset X$, \textcolor{black}{
	\[	\int_A \frac{1}{n} \sum_{l = 1}^{n} \rho^{k_l}(x) \, \mu(dx) \underset{n\to+\infty}{\longrightarrow} \int_A \rho(x) \, \mu(dx) .	\]
On the other hand $\int_A \frac{1}{n} \sum_{l = 1}^{n} \rho^{k_l}(x) \, \mu(dx)=\frac{1}{n} \sum_{l = 1}^{n}	\mu^{k_l}(A) $ converges to $\mu(A)$. 
	Hence $\int_A \rho(x) \, \mu(dx)=\mu(A)$} so that $\rho(x) = 1$, $\mu(dx)$-almost surely and
	\[	\TV\left( \frac{1}{n}\sum_{l = 1}^n \mu^{k_l}, \mu \right) = \eta^k(X) + \int_X \left\vert \frac{1}{n} \sum_{l = 1}^n\rho^{k_l}(x) - 1 \right\vert \, \mu(dx)\underset{n\to+\infty}{\longrightarrow}0.\]
	
	Finally we show ``\ref{it:KTV weak}$\implies$\ref{it:stable convergence}''. If $(\pi^k)_{k \in \N}$ does not converge stably to $\pi$, then there is a bounded Carathéodory function
	$c \colon X \times Y \to \R$, such that
	\[  \limsup_{k \to +\infty} \left\vert \int_{X\times Y} c(x,y) \, \pi^k(dx,dy) - \int_{X\times Y} c(x,y) \, \pi(dx,dy) \right\vert > 0.  \]
	
	Hence, w.l.o.g.\ there is a subsequence $(\pi^{k_j})_{j \in \N}$ such that $\pi^{k_j}(c) \geq \pi(c) + \delta$ for some $\delta > 0$. Especially, we have for any sub-subsequence $(\pi^{k_{j_i}})_{i \in \N}$ 
	of $(\pi^{k_j})_{j \in \N}$ that
	\begin{equation}\label{lowerBoundCesaroMeans}
	\frac1n \sum_{i = 1}^n \pi^{k_{j_i}}(c) \geq \pi(c) + \delta,
	\end{equation}
	whereby the Ces\`aro-means of the sub-subsequence are not stably convergent. By assumption there exists a subsequence $(\mu^{k_{j_i}})_{i\in\N}$ of $(\mu^{k_j})_{j\in\N}$ which $K$-converges in total variation to $\mu$. For $n\in\N^*$ define
	\[
	\hat\mu^n=\frac1n\sum_{i=1}^n\mu^{k_{j_i}}\quad\text{and}\quad\hat\pi^n=\frac1n\sum_{i=1}^n\pi^{k_{j_i}}.
	\]
	
	We will show that $(\hat\pi^n)_{n\in\N^*}$ converges stably to $\pi$, which will contradict \eqref{lowerBoundCesaroMeans} and end the proof. Let $\hat\mu^n(dx)=\hat\rho^n(x)\,\mu(dx)+\hat\eta^n(dx)$ be the Lebesgue decomposition of $\hat\mu^n$ w.r.t. $\mu$. Define the auxiliary sequence
	\[
	\tilde \pi^n(dx,dy) = \left((1 \wedge \hat\rho^n(x)) \, \hat\pi^n_x(dy) + (1 - \hat\rho^n(x))^+ \, \pi_x(dy)\right) \, \mu(dx).
	\]
	
	Let $c \colon X \times Y \to \R$ be Carath\'eodory and absolutely bounded by $K$, then
	\begin{align}\begin{split}\label{eq:auxiliary cara close}
	\left|\int_{X\times Y} \right. & c(x,y) \,  \left.\tilde\pi^n(dx,dy) - \int_{X\times Y} c(x,y) \, \hat\pi^n(dx,dy)\right|
	\\  
	&\leq K \left( \int_X \left|\hat\rho^n(x) - 1 \wedge \hat\rho^n(x)\right|\, \mu(dx) + 
	\int_X (1 - \hat\rho^n(x))^+ \,\mu(dx) + \hat\eta^n(X) \right)
	\\  
	&\leq K \left( \int_X \left| \hat\rho^n(x) - 1 \right| \, \mu(dx) + 2 \hat\eta^n(X)\right)
	\\  
	&\leq 2K \TV(\hat\mu^n,\mu) \underset{n\to+\infty}{\longrightarrow} 0.
	\end{split}
	\end{align}
		
\textcolor{black}{Since $(\hat \pi^n)_{n \in \N^*}$ converges to $\pi$ weakly, we deduce that so does $(\tilde \pi^n)_{n \in \N^*}$.
	Note that the first marginal $\tilde \pi^n$ is $\mu$, and therefore \cite[Lemma 2.1]{La18} or Lemma \ref{lem:convStable}} yield stable convergence of $\tilde\pi^n$ to $\pi$ as $n\to+\infty$.
	By \eqref{eq:auxiliary cara close}, we find that $(\hat\pi^n)_{n \in \N^*}$ also stably converges to $\pi$.
\end{proof}

\begin{lemma}\label{lem:caratheodorylimits}
	Let $f \colon X \to [1,+\infty)$ and $g \colon Y \to [1,+\infty)$ be continuous, and
	let $(\pi^k)_{k \in \N}$ converge to $\pi$ in $\mathcal P_{f \oplus g}(X \times Y)$.
	\begin{enumerate}[label = (\alph*)]
		\item\label{it:convlsc} If $c \colon X \times Y \to \R \cup \{ +\infty\}$ is lower semicontinuous and bounded from below by a negative multiple of $f \oplus g$, then
		\[	\liminf_{k\to+\infty} \int_{X\times Y} c(x,y) \, \pi^k(dx,dy) \geq \int_{X\times Y} c(x,y) \, \pi(dx,dy).	\]
		\item\label{it:convContinuous} If $c \colon X \times Y \to \R$ is continuous and absolutely bounded by positive multiple of $f \oplus g$, then
		\[	\lim_{k\to+\infty} \int_{X\times Y} c(x,y) \, \pi^k(dx,dy) = \int_{X\times Y} c(x,y) \, \pi(dx,dy).	\]
	\end{enumerate}

	\begin{enumerate}[label = (\alph*), resume]
		\item \label{it:convlowCaratheodory}If $c \colon X \times Y \to \R \cup \{+\infty\}$ is lower Carath\'eodory and bounded from below by a negative multiple of $f \oplus g$, and $\pi^k$ converges to $\pi$ stably, then
		\[	\liminf_{k\to+\infty}\int_{X\times Y} c(x,y) \, \pi^k(dx,dy) \geq \int_{X\times Y} c(x,y) \, \pi(dx,dy).	\]
		\item\label{it:convCaratheodory} If $c \colon X \times Y \to \R$ is Carath\'eodory and absolutely bounded by a positive multiple of $f \oplus g$, and $\pi^k$ converges to $\pi$ stably, then
		\[	\lim_{k\to+\infty} \int_{X\times Y} c(x,y) \, \pi^k(dx,dy) = \int_{X\times Y} c(x,y) \, \pi(dx,dy).	\]
	\end{enumerate}
\end{lemma}
\begin{proof}
	These results are well-known.
	Note that by \cite[Theorem 8.10.65]{Bo00} we have for every bounded lower Carathéodory map $c$ that $\pi \mapsto \pi(c)$ is lower semicontinuous 
	w.r.t.\ the topology of stable convergence.
\end{proof}

When $f\colon X \to [1,+\infty)$ and $g\colon Y \to [1,+\infty)$ are continuous functions, we say that a sequence $\pi^k \in \mathcal P_{f \oplus g}(X \times Y)$ converges stably in $\mathcal P_{f \oplus g}(X \times Y)$ to $\pi$ if and only if, for $k \to +\infty$, $\pi^k$ converges to $\pi$ in the topology of stable convergence and $\pi^k(f \oplus g)$ to $\pi(f \oplus g)$.

\begin{proposition}\label{prop: C Pf lsc}
	Let $f\colon X \to [1,+\infty)$ and $g\colon Y \to [1,+\infty)$ be continuous functions, 
	and $C\colon X \times \mathcal P_g(Y) \to\R\cup\{+\infty\}$ be measurable and bounded from below by a negative multiple of $f\oplus\hat g$. Then
	\begin{enumerate}[label=(\alph*)]
		\item\label{it:CP lsc}
		If $C$ is lower semicontinuous, then
		\begin{equation} \label{eq:Pf lsc on big space}
		\mathcal P_{f \oplus \hat g}(X \times \mathcal P_g(Y)) \ni P \mapsto \int_{X\times\mathcal P_g(Y)} C(x,p) \, P(dx,dp)
		\end{equation} 
		is lower semicontinuous.
		\item\label{it:Cpix lsc} Suppose in addition that for all $x\in X$, the map $p\mapsto C(x,p)$ is convex. Then
		\begin{equation}\label{eq:Pf lsc}
		\mathcal P_{f \oplus g}(X \times Y) \ni \pi \mapsto \int_X C(x,\pi_x)\,\mu(dx),
		\end{equation}
		where $\mu$ denotes the $X$-marginal of $\pi$, is lower semicontinuous.
		\item\label{it:CP stable}
		On the other hand, if $C$ is lower Carathéodory, then
		\begin{equation} \label{eq:Pf lsc on big space stable}
			\mathcal P_{f \oplus \hat g}(X \times \mathcal P_g(Y)) \ni P \mapsto \int_{X\times\mathcal P_g(Y)} C(x,p) \, P(dx,dp)
		\end{equation} 
		is lower semicontinuous w.r.t.\ stable convergence on $\mathcal P_{f \oplus \hat g}(X \times \mathcal P_g(Y))$.
		\item\label{it:Cpix lsc stable} Suppose in addition that for all $x\in X$, the map $p\mapsto C(x,p)$ is convex. Then
		\begin{equation}\label{eq:Pf lsc stable}
		\mathcal P_{f \oplus g}(X \times Y) \ni \pi \mapsto \int_X C(x,\pi_x)\,\mu(dx),
		\end{equation}
		where $\mu$ denotes the $X$-marginal of $\pi$, is lower semicontinuous w.r.t.\ stable convergence on $\mathcal P_{f \oplus g}(X \times Y)$.
	\end{enumerate}
\end{proposition}

\begin{proof}
	Lower semicontinuity of \eqref{eq:Pf lsc on big space} and \eqref{eq:Pf lsc on big space stable} is a consequence of Lemma \ref{lem:caratheodorylimits} \textcolor{black}{with $Y$ and $g$ replaced by $\mathcal P_g(Y)$ and $\hat g$}.
	To see \eqref{eq:Pf lsc} (resp.\ \eqref{eq:Pf lsc stable}),	
	let $(\pi_k)_{k\in\N}\in\mathcal P_{f\oplus g}(X\times Y)^\N$ converge (resp.\ stably) in $\mathcal P_{f\oplus g}(X\times Y)$ to some $\pi$. We find by the first part of Lemma \ref{lem:rel comp product space}
	an accumulation point $P \in \mathcal P_{f \oplus \hat g}( X \times \mathcal P(Y))$ of $(J(\pi^k))_{k \in \N}$.
	By possibly passing to a subsequence we can assume that $P^k := J(\pi^k)$ converges to $P$ in $\mathcal P_{f \oplus \hat g}(X \times P(Y))$ as $k$ goes to $+\infty$.
	In the case that $(\pi^k)_{k \in \N}$ converge stably in $\mathcal P_{f \oplus g}(X \times Y)$ to $\pi$, we deduce \textcolor{black}{by Lemma \ref{lem:convStable}} that $(P^k)_{k \in \N}$ converges stably in $\mathcal P_{f \oplus \hat g}(X \times \mathcal P_g(Y))$ to $P$.
	Write $\mu^k$, $k \in \N$ and $\mu$ for the $X$-marginal of $\pi^k$ and $\pi$, respectively.
	Due to \eqref{eq:Pf lsc on big space} (resp.\ \eqref{eq:Pf lsc on big space stable}), we obtain
	\begin{align*}
	\liminf_{k\to+\infty} \int_X C(x,\pi^k_x) \, \mu^k(dx) &= \liminf_{k\to+\infty} \int_{X\times\mathcal P_g(Y)} C(x,p) \, P^k(dx,dp) \\
	& \geq \int_{X\times\mathcal P_g(Y)} C(x,p)\, P(dx,dp) \\
	& \geq \int_X C\left(x, I(P_x) \right) \, \mu(dx) \\
	& = \int_X C\left(x, \hat I(P)_x\right) \, \mu(dx),
	\end{align*}
	where we used Proposition \ref{infinite-dimensional Jensen inequality} for the last inequality.
	Since $\hat I$ is continuous by Lemma \ref{lem: continuity I}, we find that $\hat I(P^k) \to \hat I(P)$ and $\hat I(P^k) = \pi^k \to \pi$ as $k \to +\infty$.
	But the weak topology is Hausdorff and therefore $\pi = \hat I(P)$ yielding
	\[ \liminf_{k\to+\infty} \int_X C(x,\pi^k_x)\, \mu^k(dx) \geq \int_X C(x,\pi_x)\, \mu(dx), \]
	and thus \eqref{eq:Pf lsc} (resp.\ \eqref{eq:Pf lsc stable}).
\end{proof}

\bibliography{../../MBjointbib/joint_biblio}
\bibliographystyle{abbrv}
\end{document}